\pdfoutput=1
\documentclass[11pt]{article}
\usepackage[margin=1in]{geometry}
\usepackage{times}
\usepackage[round]{natbib}
\usepackage{authblk}

\usepackage[utf8]{inputenc}
\usepackage[T1]{fontenc}
\usepackage{microtype}
\usepackage{amsmath,amssymb,amsfonts,amsthm,mathtools}
\usepackage{algorithm}
\usepackage{algpseudocode}
\usepackage{booktabs,array,graphicx,xcolor,hyperref,url}
\hypersetup{hidelinks,pdfauthor={First Author, Second Author, Third Author, Fourth Author, Fifth Author},pdftitle={Smooth or Separable? Sparse Block Acceleration for Entropy-Regularized Linear Programs}}

\newtheorem{theorem}{Theorem}[section]
\newtheorem{lemma}[theorem]{Lemma}
\newtheorem{corollary}[theorem]{Corollary}
\newtheorem{proposition}[theorem]{Proposition}
\newtheorem{assumption}[theorem]{Assumption}

\newcommand{\prooflink}[1]{\unskip\hfill\hyperref[#1]{\textnormal{$\blacktriangleright$}}}
\newcommand{\R}{\mathbb R}
\newcommand{\E}{\mathbb E}
\newcommand{\Prob}{\mathbb P}
\newcommand{\ip}[2]{\langle #1,#2\rangle}
\newcommand{\eps}{\varepsilon}
\newcommand{\KL}{\operatorname{KL}}
\newcommand{\gKL}{\operatorname{gKL}}
\newcommand{\nnz}{\operatorname{nnz}}
\newcommand{\supp}{\operatorname{supp}}
\newcommand{\tr}{\operatorname{tr}}
\newcommand{\Diag}{\operatorname{Diag}}
\newcommand{\Cov}{\operatorname{Cov}}
\newcommand{\Var}{\operatorname{Var}}
\newcommand{\se}{{\rm se}}
\newcommand{\lse}{{\rm lse}}
\newcommand{\nB}{N_{\mathcal B}}  

\title{Smooth or Separable?\\Sparse Block Acceleration for\\Entropy-Regularized Linear Programs}
\author[1]{Irina Podlipnova}
  \author[1]{Maxim Mashtaler}
  \author[2]{Artem Agafonov}
  \author[3]{Yuriy Dorn}
  \author[4]{Alexander Gasnikov}                                                                                    
  \affil[1]{Moscow Independent Research Institute of Artificial Intelligence}        
  \affil[2]{Mohamed bin Zayed University of Artificial Intelligence}
  \affil[3]{Lomonosov Moscow State University}
  \affil[4]{Innopolis University}
  \affil[ ]{\texttt{podlipnova.i@miriai.org}, \texttt{mashtaler.maxim@gmail.com}, \texttt{artem.agafonov@mbzuai.ac.ae}, \texttt{dornyv@my.msu.ru}, \texttt{gasnikov@yandex.ru}}
\date{}

\begin{document}
\maketitle

\begin{abstract}
We study entropy-regularized linear programs with sparse affine constraints and compare two exact
dual representations induced by whether a redundant normalization constraint is retained or
eliminated. Eliminating it yields a partially separable sum-exp dual with sparse affine factors but
unbounded curvature. Our main result is an accelerated randomized block method whose acceleration
preserves sparse data-access cost. The key ingredient is a global entropy-specific Hessian--gap
bound that remains valid for signed constraint matrices, admits computable block-overlap
refinements, and provides deterministic curvature certificates before sampling. This geometry
supports square-root importance sampling and an exact lazy implementation in which sparse
arithmetic cost is weighted by local curvature rather than by a worst-block factor. We further
introduce a gap-dependent curvature dimension that connects the generalized-smooth regime to the
classical effective-rank intuition for coordinate methods. Retaining the normalization constraint
yields a log-sum-exp dual with globally bounded higher derivatives, enabling gradient-regularized
and cubic-regularized Newton methods. Experiments compare the resulting methods in terms of
coordinate epochs, arithmetic operations, and wall-clock time, separating the effects of
acceleration, touched-data sparsity, and problem-specific structure.
\end{abstract}

\section{Introduction}
Entropy-regularized linear programs (ELPs) with sparse constraints,
\begin{equation}\label{eq:primal-short}
\min_{x\ge0}\left\{\ip{c}{x}+\gamma\sum_jx_j(\log x_j-1):Ax=b,\ \mathbf1^\top x=1\right\},
\qquad A\in\R^{m\times n},\ n\gg m,
\end{equation}
carry a redundant simplex constraint: the affine system already fixes the total mass, in the sense
that some $q\in\R^m$ has $A^\top q=\mathbf1$ and $\ip{q}{b}=1$. Discrete optimal transport is the
standard example -- summing the row marginals makes one of the $m$ equations redundant -- and
entropic regularization is what made large instances tractable, by turning them into a matrix
scaling that alternating projection solves \citep{sinkhorn1967concerning,cuturi2013sinkhorn,altschuler2017near}.
The matrix is sparse: for transport it is an incidence matrix with two nonzeros per column.

\paragraph{One primal problem, two duals.}
The redundancy gives \emph{two} exact duals, on opposite sides of a trade. Keeping the simplex
equation and dualizing gives the \emph{log-sum-exp} objective; removing the redundant equation first
gives the \emph{sum-exp} objective:
\begin{align}
\phi_\lse(\lambda)&=-\ip{b}{\lambda}+\gamma\log\sum_{j=1}^n\exp\bigl(\tfrac{a_j^\top\lambda-c_j}{\gamma}\bigr),\label{eq:lse-short}\\
\phi_\se(\lambda)&=-\ip{b}{\lambda}+\gamma\sum_{j=1}^n\exp\bigl(\tfrac{a_j^\top\lambda-c_j}{\gamma}\bigr).\label{eq:se-short}
\end{align}
The logarithm normalizes the primal response onto the simplex, which bounds every derivative of
$\phi_\lse$ -- Nesterov smoothing \citep{nesterov2005smooth} -- but couples all $n$ factors through
their sum, so an accelerated update is not local. Without it $\phi_\se$ is a sum of exponentials of sparse affine
functions: factor $j$ depends only on the support of column $j$, so a coordinate update reads only
incident data. What $\phi_\se$ loses is a global curvature bound; its Hessian grows without limit and
it is smooth only in the generalized sense, with curvature controlled by the optimality gap.

\paragraph{Why one wants the separable side, and what blocks it.}
The reason is spectral: when curvature is spread unevenly across coordinates, square-root importance
sampling improves the complexity bound of a full-gradient method by up to $\sqrt{m/d_H}$, with $d_H=\tr H/\|H\|_2$ the effective rank of the
Hessian \citep{allenzhu2016,nesterovstich2017,gasnikov2015accelerated,tropp2015}. That is $\sqrt m$ only when one
eigenvalue dominates, and it is not: on the family we test the factor sits at $4.2$ throughout
$m\in[100,1000]$, and on transport between $1.6$ and $3.4$ (Section~\ref{sec:geometry}). The gain is thus moderate and concerns data access.

Each route then gives up one of the two properties that matter. On $\phi_\lse$, gradient and
second-order steps read all $n$ factors or cost $O(m^3)$, and an accelerated coordinate step must
re-evaluate the normalizing sum at the extrapolated point, so sparsity is lost. On $\phi_\se$ sparsity survives, and
acceleration is available in the abstract gap-dependent class: the coordinate method of
\citet{lobanov2024} is not accelerated, and \citet{lobanov2026}, concurrent with this work,
accelerates coordinate methods under an $(H_0,H_1)$ condition at the cost of an exact one-dimensional
relaxation per update; neither accounts for the arithmetic of an update. The
obstruction that remains is concrete: accelerated coordinate methods fix their sampling law and step
sizes from curvature constants known \emph{before} a block is drawn
\citep{nesterov2012coordinate,leesidford2013,fercoq2015,allenzhu2016,nesterovstich2017,gasnikov2015accelerated}, and $\phi_\se$ has no finite
such constants, while re-estimating curvature at the iterate needs the gap, which touches all $n$
factors and destroys the sparse access one came for. Replacing a constant by a function value is the
generalized-smoothness move \citep{zhang2020clipping,crawshaw2022,li2023generalized,vankov2025,gorbunov2025methods};
what it lacks is the instantiation -- constants read off $A$, $b$ and $\gamma$, and an accounting
that charges the nonzeros an update touches (Appendix~\ref{app:resource-comparison}).

Second-order methods do not separate the two duals. Both are quasi-self-concordant --
$M_{\rm qsc}\le\Delta_A/\gamma$ for $\phi_\lse$ and $\rho/\gamma$ for $\phi_\se$
(Proposition~\ref{prop:se-qsc}) -- so gradient-regularized Newton (GRN) \citep{doikov2023qsc} runs on
either; only $\phi_\lse$ has the globally bounded third derivative on which cubic-regularized
Newton (CRN) \citep{nesterov2006cubic} and the accelerated order-$2$ tensor method
\citep{dvurechensky2024near}, with rate $\eps^{-1/3}$, rely. The choice in our title is therefore
between two kinds of \emph{step}, a global one costing an $m\times m$ solve or a coordinate one reading
only incident data, and since both run on $\phi_\se$, Section~\ref{sec:lse-short} compares them on one
dual. Transport already uses both
\citep{dvurechensky2018computational,guo2020fast,brauer2017sinkhornnewton,tang2024sns,tang2024ssns}.

\paragraph{Smooth or separable?}
Smoothness is the property to exploit when the dual dimension is small
enough for an $m\times m$ factorization and high accuracy is required: a second-order step costs
$O(m^3)$, but its iteration count grows only as $\log\frac1\eps$, and on our synthetic family this
keeps the wall-clock advantage up to $m$ between $400$ and $1000$. Separability is the property to
exploit when $m$ is larger and $A$ is sparse: a coordinate step then costs only the data it reads, and
beyond that crossover the coordinate methods are ahead (Table~\ref{tab:crossover}). Separability alone
does not suffice when the blocks carry structure of their own: on transport the smooth route beats
generic coordinate methods, and Sinkhorn, which minimizes exactly over a whole marginal block, is
faster than both (Section~\ref{sec:exp-short}).

The technical core is a curvature bound for $\phi_\se$ that holds everywhere and is known before
sampling. Writing $F:=\phi_\se-\phi_\se^\star$ for the dual gap and fixing a partition
$\mathcal B$ of the $m$ dual coordinates into blocks, we prove that every block $B$ admits
\begin{equation}\label{eq:intro-main}
\nabla^2_{BB}\phi_\se(\lambda)\preceq\bigl(\kappa_B+\nu_BF(\lambda)\bigr)I_B
\qquad\text{for every }\lambda,
\end{equation}
where $\nabla^2_{BB}$ is the diagonal block of the Hessian indexed by $B$, $\kappa_B$ the curvature that survives
at a solution and $\nu_B$ the rate at which it grows with the gap; Theorem~\ref{thm:hg-short} gives
both, and Section~\ref{sec:computable-geometry} gives bounds computable from $A$, $b$ and $\gamma$.
Neither constant depends on which block is drawn, so two sampling tables are built once, the step
sizes follow a deterministic schedule, and the accelerated triple $(u_k,v_k,y_k)$ can be stored so
that one sampled block costs $d_B:=|B|+\nnz(A_{B:})$ operations and no other coordinate is touched.
A block is a set of dual coordinates drawn together, and the partition is an input: singletons
are the default, while for transport the natural coarser choice is the two marginal families, the
source rows and the target rows, on which exact block minimization is the Sinkhorn half-step
(Section~\ref{sec:exp-short}).

\paragraph{Contributions.}\label{sec:intro-contrib}
We give an accelerated randomized block coordinate method for $\phi_\se$, so that acceleration and
sparse data access hold at once, and we build it on three things specific to the entropy dual rather
than to a smoothness abstraction. First, \eqref{eq:intro-main} itself: its constants are computable
from the data, it survives signed rows without any knowledge of the solution, and its growth
coefficient is optimal in its class up to the factor $\theta$. Second, that those constants are fixed \emph{before} the
block is drawn, which is what makes the lazy sparse representation exact and the second-moment terms
of the accelerated analysis cancel. Third, an arithmetic accounting in which the data a block reads
is weighted by that block's own curvature rather than by a worst-block factor;
Appendix~\ref{app:heterogeneous-cost} exhibits a family on which the two accountings differ
polynomially. The leading arithmetic cost is
$\widetilde O\bigl(r(\gamma\eps)^{-1/2}\sum_Bd_B\sqrt{\gamma\kappa_B}\bigr)$: every block contributes
the data it reads, $d_B$, weighted by its own curvature, and $\gamma\kappa_B$ carries no $\gamma$.
Here $r$ bounds the distance from the starting point to a dual solution.

\section{Two exact duals}\label{sec:duals-short}

Assume strict primal feasibility and the redundancy certificate $A^\top q=\mathbf1$,
$\ip{q}{b}=1$, and write $x_j(\lambda)=\exp((a_j^\top\lambda-c_j)/\gamma)$,
$Z(\lambda)=\sum_jx_j(\lambda)$ and $p_j=x_j/Z$ for the unnormalized and normalized responses;
Appendix~\ref{app:duals} states the assumption formally and collects the moment identities of the
two duals.
Keeping the simplex constraint and dualizing gives \eqref{eq:lse-short}, with softmax primal
response $p(\lambda)$; dropping the redundant equation first gives \eqref{eq:se-short}, whose
response is the same expression without the normalizing denominator. The two are related by an
exact one-dimensional minimization along $q$. Neither is new: \citet{genevay2016stochastic} state
both for transport, the unnormalized one as the dual and the log-sum-exp one as the semi-dual, and run
SAG on the first and SGD on the second. Which one an \emph{accelerated block} method wants is open.

\begin{proposition}[Exact link]\label{prop:dual-short}
Let $Z(\lambda)=\sum_jx_j(\lambda)$. Then
\[
\min_t\phi_\se(\lambda+tq)=\phi_\lse(\lambda)+\gamma,
\qquad t^\star(\lambda)=-\gamma\log Z(\lambda),
\]
and $x(\lambda+t^\star q)=p(\lambda)$. Writing $F_\se:=\phi_\se-\phi_\se^\star$ and
$F_\lse:=\phi_\lse-\phi_\lse^\star$,
\begin{align}
F_\se(\lambda)&=F_\lse(\lambda)+\gamma\bigl(Z(\lambda)-1-\log Z(\lambda)\bigr),\label{eq:gap-split}\\
F_\lse(\lambda)&=\gamma\KL(x^\star\|p(\lambda)).\label{eq:lse-kl}
\end{align}
\end{proposition}

\noindent Restricted to the line $\lambda+tq$, \eqref{eq:se-short} becomes
$-t+\gamma e^{t/\gamma}Z(\lambda)$ up to a constant, and its stationary point gives the claim
(Appendix~\ref{app:duals}).

Equation~\eqref{eq:gap-split} separates two different errors. The LSE gap is exactly the
Kullback--Leibler divergence to the optimal primal point; the SE gap is that divergence \emph{plus}
a penalty $\gamma(Z-1-\log Z)$ for the response carrying the wrong total mass, which vanishes only
at $Z(\lambda)=1$. The two are therefore not interchangeable as accuracy metrics, and an experiment
must say which it reports.

\section{Curvature of the SE dual}\label{sec:hg-short}
Fix a partition $\mathcal B$ of the $m$ dual coordinates. We write
$H_{BB}(\lambda)=\nabla^2_{BB}\phi_\se(\lambda)$ for the block Hessian,
$H^\star_{BB}:=H_{BB}(\lambda^\star)$ for its value at a dual solution, $F:=\phi_\se-\phi_\se^\star$
for the dual gap, and $\rho_B=\max_j\|A_{B,j}\|_2$ for the largest norm of a column restricted to
$B$. Dual solutions need not be unique, but $H^\star$ is: every $\lambda^\star$ gives the same
primal response $x^\star$, which strict convexity makes unique.

\subsection{The Hessian--gap inequality}

\begin{theorem}[Entropy Hessian--gap domination]\label{thm:hg-short}
For every $\theta>1$, every block $B$, and every $\lambda$,
\begin{equation}\label{eq:hg-short}
H_{BB}(\lambda)\preceq c_\theta H^\star_{BB}+\frac{\theta\rho_B^2}{\gamma^2}F(\lambda)I_B,
\qquad c_\theta=\theta\log\tfrac{\theta}{\theta-1}.
\end{equation}
In particular, for any $\bar H_B\ge\|H^\star_{BB}\|_2$ the choice
$\kappa_B=2\log2\,\bar H_B$ and $\nu_B=2\rho_B^2/\gamma^2$ satisfies \eqref{eq:intro-main}.
No sign assumption on $A$ is required.
\end{theorem}

Both terms come from a single scalar inequality. Write $x_j(\lambda)=x_j^\star e^{t_j}$ with
$t_j=a_j^\top(\lambda-\lambda^\star)/\gamma$. The gap then has the exact Bregman form
$F(\lambda)=\gamma\sum_jx_j^\star h(t_j)$, $h(t)=e^t-1-t$, while the block Hessian quadratic form
at $u$ is $\gamma^{-1}\sum_jx_j^\star e^{t_j}(A_{B,j}^\top u)^2$. The Hessian is weighted by $e^t$
and the gap by $h(t)$, and the elementary inequality $e^t\le\theta h(t)+c_\theta$ converts one into
the other term by term; the gap-dependent piece is exactly the price of that conversion, and
$\theta=2$ gives $c_2=2\log2$. No sign condition on $A$ is needed, because what is bounded is the square
$(A_{B,j}^\top u)^2$, and no bounded set, because the right-hand side grows with $F(\lambda)$. The
growth coefficient is not improvable: any bound
$\nabla^2_{BB}\phi_\se(\lambda)\preceq\Psi_B+\nu_BF(\lambda)I_B$ valid for every $\lambda$ forces
$\nu_B\ge\rho_B^2/\gamma^2$, so \eqref{eq:hg-short} is optimal in its class up to $\theta$
(Proposition~\ref{prop:nu-tight}, Appendix~\ref{app:hg-proof}).

\subsection{Computing the constants}\label{sec:computable-geometry}
Theorem~\ref{thm:hg-short} is stated via $H^\star$, which we do not know, but the method needs only
a valid upper bound $\bar H_B$, and several are available without solving the problem. Universally
$\bar H_B=\rho_B^2/\gamma$, since $\sum_jx_j^\star=1$. If the rows are sign-oriented so that $A\ge0$, then
$\sum_jA_{ij}x_j^\star=b_i$ is known exactly and a Cauchy--Schwarz step over the nonzeros of each
block column supplies a second bound, so one may take
\begin{equation}\label{eq:h-overlap-short}
\bar H_B=\tfrac1\gamma\min\Bigl\{\rho_B^2,\ \max_{i\in B}\bigl[b_i\max_{j:A_{ij}\ne0}\omega_{Bj}A_{ij}\bigr]\Bigr\},
\quad \omega_{Bj}=|\{i\in B:A_{ij}\ne0\}| ,
\end{equation}
where the overlap count $\omega_{Bj}$ says how many rows of the block a single factor couples. Neither
term dominates: for singleton blocks $\omega=1$ and the second always wins, collapsing to
$\bar H_i=\|A_{i:}\|_\infty b_i/\gamma$, whereas for a block whose columns concentrate their mass on
one row the factor $\omega_{Bj}$ can make it the larger by $\omega_{\max}=\max_j\omega_{Bj}$. For signed rows the same idea works from the row mean and
the row range alone.

\begin{corollary}[Range--moment bound for signed rows]\label{cor:range-moment}
For row $i$ let $\alpha_i^-:=\min_jA_{ij}$ and $\alpha_i^+:=\max_jA_{ij}$, counting the implicit
zeros of sparse storage. Define
\begin{equation}\label{eq:h-range}
\mu_i:=(\alpha_i^-+\alpha_i^+)b_i-\alpha_i^-\alpha_i^+,\qquad
\bar H^{\rm rng}_i:=\mu_i/\gamma .
\end{equation}
Then $0\le H_{ii}^\star\le\bar H^{\rm rng}_i\le\|A_{i:}\|_\infty^2/\gamma$, so
$\kappa_i=2\log2\,\bar H^{\rm rng}_i$ and $\nu_i=2\|A_{i:}\|_\infty^2/\gamma^2$ are computable
without sign orientation and without knowing $x^\star$. If $\alpha_i^-=0$ this recovers
$\bar H_i=\|A_{i:}\|_\infty b_i/\gamma$. \prooflink{proof:range-moment}
\end{corollary}

The proof is one line, from $(a-\alpha_i^-)(a-\alpha_i^+)\le0$ averaged against $x^\star$. Outside one
structural subclass these bounds can be arbitrarily loose; nothing is lost on scaled indicator rows,
$A_{ij}\in\{0,\rho_i\}$, where $A_{ij}^2=\rho_iA_{ij}$ collapses the row's second moment to its mean
and feasibility pins that to $b_i$, giving $H^\star_{ii}=\rho_ib_i/\gamma$ exactly. Incidence matrices, and hence transport, are of
this form (Appendix~\ref{app:geometry-proofs}). Summed over singleton blocks, the two constants
define the gap-dependent curvature dimension, which interpolates between the effective rank of
$H^\star$ and the spread of the row magnitudes (Section~\ref{sec:geometry}).

\section{Accelerated block coordinate descent on the SE dual}\label{sec:alg-short}
The method applies to any convex $C^2$ function $f:\R^m\to\R$ with an attained minimum satisfying
the block gap-growth (BGG) condition
\begin{equation}\label{eq:bgg}
0\preceq\nabla^2_{BB}f(u)\preceq(\kappa_B+\nu_BF(u))I_B ,\qquad F=f-f^\star,
\end{equation}
globally, with finite $\kappa_B,\nu_B\ge0$; Theorem~\ref{thm:hg-short} supplies it for $\phi_\se$.
Fix any $u^\star\in\arg\min f$, write $f^\star=f(u^\star)$, let $\mathcal F_k$ be the
$\sigma$-algebra of the first $k$ block draws, and let $U_B:\R^{|B|}\to\R^m$ be the zero-padding
injection, so that $U_B\nabla_Bf(u)$ is the block gradient viewed in $\R^m$. Blocks with $\kappa_B=\nu_B=0$ are dropped as exact
(Appendix~\ref{app:acc-proof}).

Set $C_0=\sum_B\sqrt{\kappa_B}$ and $C_1=\sum_B\sqrt{\nu_B}$: $C_0$ multiplies the
accuracy-dependent term and $C_1$, which measures how fast curvature inflates, carries no negative
power of $\eps$. We assume $C_0>0$; otherwise $F\equiv0$ by a Gronwall argument.

Ideally $B$ is sampled with probability proportional to $\sqrt{\kappa_B+5\nu_B\Delta_k}$, where
$\Delta_k$ is the gap level the potential currently guarantees (Lemma~\ref{lem:presampling}). The
separable surrogate $\sqrt{\kappa_B}+\sqrt{5\nu_B\Delta_k}$ loses at most a factor $\sqrt2$ and is a
mixture of \emph{two fixed distributions} $\pi^{(0)}_B\propto\sqrt{\kappa_B}$ and
$\pi^{(1)}_B\propto\sqrt{\nu_B}$, so two alias tables built once in linear time
\citep{vose1991linear} give $O(1)$ expected sampling with no per-iteration scan.

\begin{algorithm}[t]
\caption{Accelerated block step with precomputed curvature bounds}
\label{alg:main-short}
\small
\begin{algorithmic}[1]
\Require anchor $z$; $\bar\Delta\ge F(z)$; $r\ge\|z-u^\star\|$; failure level $\delta$; target $\eps$; alias tables for $\pi^{(0)}$ and, if $C_1>0$, $\pi^{(1)}$.
\State If $C_0=0$, $r=0$, $\bar\Delta=0$, or $\eps\ge\bar\Delta$, \Return $z$.
\State $T_0\gets r^2/(2\bar\Delta)$, $u_0=v_0\gets z$, $\Pi\gets r^2/\delta$.
\For{$k=0,1,\ldots$}
  \State $\Delta_k\gets\Pi/T_k$ \Comment{gap level the potential guarantees}
  \State $M_k\gets C_0+\sqrt{5\Delta_k}\,C_1$
  \State choose $a_{k+1}>0$ with $M_k^2a_{k+1}^2=T_k+a_{k+1}$; \ $T_{k+1}\gets T_k+a_{k+1}$
  \State $y_k\gets(T_ku_k+a_{k+1}v_k)/T_{k+1}$
  \State with probability $C_0/M_k$ draw $B_k\sim\pi^{(0)}$, else draw $B_k\sim\pi^{(1)}$
  \State $\sqrt{\bar L_{k,B_k}}\gets\sqrt{\kappa_{B_k}}+\sqrt{5\nu_{B_k}\Delta_k}$; \ $p_{k,B_k}\gets\sqrt{\bar L_{k,B_k}}/M_k$
  \State $g\gets\nabla_{B_k}f(y_k)$; \ pick $u_{k+1}\in y_k+\operatorname{range}U_{B_k}$ with $f(u_{k+1})\le f(y_k-U_{B_k}g/\bar L_{k,B_k})$
  \State $v_{k+1}\gets v_k-(a_{k+1}/p_{k,B_k})U_{B_k}g$
  \If{$T_{k+1}\ge\Pi/\eps$} \State \Return $u_{k+1}$ \EndIf
\EndFor
\end{algorithmic}
\end{algorithm}

Algorithm~\ref{alg:main-short} realizes the mixture exactly: every block's curvature bound is
determined before the draw, and only the sampled block's constant and gradient are evaluated. The
constant $5$ is a safety margin over $F(y_k)\le4\Delta_k$. The explicit step always satisfies line~10, and
so does exact block minimization; on the event $\Phi_k\le\Pi$ (the potential defined below),
Lemma~\ref{lem:presampling} turns line~10 into the decrease $\|g\|_2^2/(2\bar L_{k,B_k})$.

A bound on $F(u_k)$ alone says nothing about the curvature at the extrapolated point $y_k$; the
potential controls both the gap at $u_k$ and the displacement of $v_k$ from a minimizer.

\begin{lemma}[Pre-sampling block model]\label{lem:presampling}
Assume \eqref{eq:bgg}. Let $T,a,\Pi>0$, put $T_+=T+a$ and $\Delta=\Pi/T$, and let
$M=C_0+\sqrt{5\Delta}\,C_1$ satisfy $M^2a^2=T_+$. If
\begin{equation}\label{eq:presampling-potential}
TF(u)+\tfrac12\|v-u^\star\|_2^2\le\Pi,\qquad y:=\frac{Tu+av}{T_+},
\end{equation}
then $F(y)\le4\Delta$. Consequently every block bound
$\sqrt{\bar L_B}:=\sqrt{\kappa_B}+\sqrt{5\nu_B\Delta}$ is valid at $y$,
\begin{equation}\label{eq:presampling-model}
\nabla^2_{BB}f(y)\preceq\bar L_BI_B ,
\end{equation}
and $y-U_B\nabla_Bf(y)/\bar L_B$ decreases $f$ by at least
$\|\nabla_Bf(y)\|_2^2/(2\bar L_B)$.
\prooflink{proof:presampling}
\end{lemma}

\noindent Every quantity in \eqref{eq:presampling-model} is $\mathcal F_k$-measurable at iteration
$k$, so the whole sampling law of line~8 is determined by the potential rather than by the draw.

\begin{theorem}[One accelerated phase]\label{thm:phase-short}
Assume \eqref{eq:bgg} with $C_0>0$, and let the inputs of Algorithm~\ref{alg:main-short} satisfy
$\bar\Delta\ge F(z)$, $r\ge\|z-u^\star\|_2$ for some minimizer $u^\star$, $\delta\in(0,1)$ and
$0<\eps<\bar\Delta$. Let $N$ be the first index with $T_N\ge\Pi/\eps$. Then
$\Prob(F(u_N)\le\eps)\ge1-\delta$ and
\begin{equation}\label{eq:query-short}
N=\widetilde O\!\left(1+\frac{rC_0}{\sqrt{\delta\eps}}+\frac{rC_1}{\sqrt\delta}\right).
\end{equation}
\prooflink{proof:phase-short}
\end{theorem}

The proof runs the potential $\Phi_k=T_kF(u_k)+\tfrac12\|v_k-u^\star\|^2$. On the event
$\Phi_k\le\Pi$, the point $y_k$ lies in $\{F\le5\Delta_k\}$, where the block model is valid, and this
is known before the draw, so no segment solve or gap evaluation is needed. The quadratic terms then
cancel because $a_{k+1}^2/p_{k,B}=T_{k+1}p_{k,B}/\bar L_{k,B}$, which is how $a_{k+1}$ is chosen in
line~6, and Doob's maximal inequality for the stopped supermartingale turns $\Phi_0\le r^2=\delta\Pi$
into the failure bound (Appendix~\ref{app:acc-proof}). Since the schedule $T_k$ does not depend on
the draws, $N$ is deterministic and the expected work is $\sum_{k<N}\E[d_{B_k}]$.

\paragraph{Arithmetic cost.}
Let $d_B$ be the cost of touching block $B$ and set $S_0=\sum_Bd_B\sqrt{\kappa_B}$,
$S_1=\sum_Bd_B\sqrt{\nu_B}$: the same two channels, weighted by the data each block reads. Low
effective dimension and low access cost are different advantages, and the pair
$(C_\bullet,S_\bullet)$ keeps them apart (Table~\ref{tab:constants-main},
Appendix~\ref{app:acc-proof}).

\begin{theorem}[Sparse arithmetic work]\label{thm:work-short}
Assume the hypotheses of Theorem~\ref{thm:phase-short} for $\phi_\se$, with $C_0>0$ and
$0<\eps<\bar\Delta$. Run
$J_{\rm rep}:=\lceil\log(2\bar\Delta/\eps)/\log3\rceil$ independent copies of
Algorithm~\ref{alg:main-short} from the \emph{same anchor} $z$, each with failure level $1/3$ and
target $\eps/2$, and return the point $\widehat u$ of smallest objective among $z$ and the outputs.
Then $\E F(\widehat u)\le\eps$, and with
$\chi:=\min\{\sqrt{2\bar\Delta}/C_0,\,1/C_1\}$ for $C_1>0$ and $\chi:=0$ otherwise,
\begin{equation}\label{eq:work-short}
\E W_{\rm opt}=\widetilde O\!\left(\frac{S_0}{C_0}+\frac{rS_0}{\sqrt\eps}+rS_1+\chi S_1\right),
\end{equation}
including phase-end comparisons, best-candidate storage and resets, on top of
$W_{\rm init}=O(\nnz(A)+m+n)$ for anchor projections, constants and alias tables, and again for a
full primal response.
\end{theorem}

The repetitions turn the $1/\sqrt\delta$ of \eqref{eq:query-short} into a logarithm; the anchor
stays fixed because a better objective value at a new anchor does not imply a smaller distance to
$u^\star$; restarting from new anchors instead gives the same guarantee with the sublevel diameter in
place of $r$ (Proposition~\ref{prop:restart}). The same repetitions also give
$\Prob(F(\widehat u)\le\eps/2)\ge1-3^{-J_{\rm rep}}$.

Only $rS_0/\sqrt\eps$ carries a negative power of $\eps$, while $rS_1$ scales as $1/\gamma$: on
the balanced family of Proposition~\ref{prop:balanced-family} the ratio is
$\sqrt{m\eps/(\gamma s_{\rm col}\log2)}$, so the growth term takes over once
$\eps\gtrsim\gamma s_{\rm col}/m$, and the leading cost quoted in Section~\ref{sec:intro-contrib}
is the small-$\eps$ face of Theorem~\ref{thm:work-short}. For signed rows
Corollary~\ref{cor:range-moment} replaces $\kappa_i$ by the range--moment constant throughout.

\subsection{The lazy sparse representation}\label{sec:lazy}
The representation below is available only because the constants are fixed in advance. Inside one
phase, write $v_k=z+\tilde v_k$
and $w_k=T_k(u_k-v_k)$; the interpolation point then factors as
$y_k=z+\tilde v_k+w_k/T_{k+1}$, and the two state vectors obey, with
$\xi=a_{k+1}/p_{k,B_k}$,
\begin{equation}\label{eq:lazy-updates}
\tilde v_{k+1}=\tilde v_k-\xi U_{B_k}g,\qquad
w_{k+1}=w_k+T_{k+1}\bigl(\xi U_{B_k}g+\eta_k\bigr),\qquad
\eta_k:=u_{k+1}-y_k .
\end{equation}
The displacement $\eta_k$ is supported on $B_k$ by the restriction in line~10, and equals
$-U_{B_k}g/\bar L_{k,B_k}$ for the explicit step, so both updates change only the sampled
coordinates whichever admissible $u_{k+1}$ is taken. Maintaining the three projections
$\zeta^0=A^\top z$, $\zeta^v_k=A^\top\tilde v_k$ and $\zeta^w_k=A^\top w_k$ lets the sampled
gradient be read off as
\begin{equation}\label{eq:sparse-grad}
(\nabla_{B_k}\phi_\se(y_k))_i+b_i
=\sum_{j:A_{ij}\ne0}A_{ij}\exp\bigl(\tfrac{\zeta^0_j+\zeta^v_{k,j}+\zeta^w_{k,j}/T_{k+1}-c_j}{\gamma}\bigr),
\end{equation}
touching only the rows of $B_k$ and the primal factors incident to them. Each projection update is
supported on $\supp(A_{B_k:})$, so one sampled block costs $O(d_B)$ with
$d_B=|B|+\nnz(A_{B:})$.
Phase safeguards are sparse too: the objective difference from the anchor is evaluable over the
visited coordinates alone, so accepts, rejects and resets add no hidden $O(m)$ or $O(n)$ rescan
(Appendix~\ref{app:sparse-alg}).

\section{The LSE route}\label{sec:lse-short}
Retaining the normalization constraint destroys separability and buys global smoothness of every
order. With $p(\lambda)$ the softmax response,
\[
\nabla\phi_\lse(\lambda)=Ap(\lambda)-b,\qquad
\nabla^2\phi_\lse(\lambda)=\tfrac1\gamma A(\Diag p-pp^\top)A^\top,
\]
and two things change. The gradient norm \emph{is} the primal infeasibility of the normalized
response; $\nabla\phi_\se=Ax(\lambda)-b$ is the residual of the unnormalized one, so the SE
\emph{coordinate} method has a stopping criterion but no cheap way to maintain it. And the Hessian is
a covariance rather than a second moment, so every derivative is bounded on the subspace $\mathcal H$ orthogonal to the flat directions of
$\phi_\lse$, and each fixed order $\ell$ is Lipschitz with $O_\ell(\rho^{\ell+1}/\gamma^\ell)$
(Appendix~\ref{app:tensor}).

\begin{theorem}[Global LSE geometry]\label{thm:lse-geometry}
Let $\Delta_A:=\max_{j,k}\|a_j-a_k\|_2$. Then for every $\lambda$ and all $u,v,w$,
\begin{equation}\label{eq:lse-constants}
\begin{gathered}
\|\nabla^2\phi_\lse(\lambda)\|_2\le\frac{\Delta_A^2}{4\gamma},\qquad
|\nabla^3\phi_\lse(\lambda)[u,v,w]|\le\frac{\Delta_A^3}{4\gamma^2}\|u\|_2\|v\|_2\|w\|_2,\\
|\nabla^3\phi_\lse(\lambda)[u,u,v]|\le\frac{\Delta_A}{\gamma}\bigl(u^\top\nabla^2\phi_\lse(\lambda)u\bigr)\|v\|_2 .
\end{gathered}
\end{equation}
So $\phi_\lse$ has $L_\lse\le\Delta_A^2/(4\gamma)$, is quasi-self-concordant with
$M_{\rm qsc}\le\Delta_A/\gamma$, and has $M_2\le\Delta_A^3/(4\gamma^2)$.
\prooflink{proof:lse-geometry}
\end{theorem}

\noindent Since $\Delta_A\le2\rho$ with $\rho=\max_j\|a_j\|_2$, the coarser column-norm constants
$L_\lse\le\rho^2/\gamma$, $M_{\rm qsc}\le2\rho/\gamma$ and $M_2\le2\rho^3/\gamma^2$ follow and are
what the exponent comparisons use. This is Nesterov smoothing \citep{nesterov2005smooth}. A
single non-accelerated LSE coordinate update can be sparse if the weights and their sum are
maintained incrementally; an accelerated one cannot, since the sum is needed at the extrapolated
point. The property that decides which dual a second-order method needs is the following.

\begin{proposition}[The SE dual is quasi-self-concordant]\label{prop:se-qsc}
For every $\lambda$ and all $u,v$,
\[
\bigl|\nabla^3\phi_\se(\lambda)[u,u,v]\bigr|
\le\frac{\max_j|a_j^\top v|}{\gamma}\,\nabla^2\phi_\se(\lambda)[u,u]
\le\frac{\rho\|v\|_2}{\gamma}\,\nabla^2\phi_\se(\lambda)[u,u],
\]
so $M_{\rm qsc}(\phi_\se)\le\rho/\gamma$, no worse than the column-norm bound $2\rho/\gamma$ for
$\phi_\lse$. \prooflink{proof:se-qsc}
\end{proposition}

\noindent Gradient-regularized Newton therefore applies to $\phi_\se$ directly, with iterates
confined to $z+\operatorname{range}A$ where the SE sublevel sets are bounded
(Appendix~\ref{app:input-certificates}); matrix scaling is Doikov's own worked example of a
quasi-self-concordant objective of this form \citep{doikov2023qsc}. What $\phi_\se$
lacks is a global Lipschitz Hessian, so tensor and gradient-norm methods still
require $\phi_\lse$. Neither method we run on this dual is ours; Theorem~\ref{thm:lse-geometry}
contributes the moment representation and the constants that place the dual inside their theory.
GRN's shift $\sigma_k\|g_k\|_2I$ repairs the gauge degeneracy, so quasi-self-concordance alone
gives a rate on the whole sublevel set without a cubic subproblem
\citep{doikov2023qsc,doikov2024gradient}; accelerated alternating minimization
\citep{guminov2021combination} needs blocks with a closed form, which general sparse $A$ does not
give. Appendices~\ref{app:tensor} and~\ref{app:resource-comparison} carry the gauge
construction, Proposition~\ref{prop:grn-rate} and the comparisons at a common tolerance.

\paragraph{Cost as a function of the dual dimension.}
A coordinate step reads one row at $O(d_i)$, while a second-order step assembles the Hessian in
$O(\nnz(A)\omega_{\max}+m^2)$ and factorizes it in $O(m^3)$, with an iteration count almost
independent of accuracy and dimension; the same GRN runs on $\phi_\se$
(Proposition~\ref{prop:se-qsc}) at a fifth more iterations and a third more arithmetic. The fitted
exponents of Table~\ref{tab:crossover} and Figure~\ref{fig:crossover} give two crossings: the
second-order branch loses its arithmetic advantage between $m=100$ and $200$ but keeps its
wall-clock advantage until between $m=400$ and $1000$, because an $m^3$ factorization runs near peak
throughput while a coordinate sweep is a scattered gather over $\nnz(A)$. By Pinsker, a dual gap
$\eps=\gamma\eps_{\rm p}^2/(2\rho^2)$ gives $\E\|A\widehat x-b\|_2\le\eps_{\rm p}$, in expectation
(Appendix~\ref{app:cert-ot}).

\section{Experiments}\label{sec:exp-short}
The experiments separate acceleration together with curvature-weighted sampling (ACD against RCD
in coordinate epochs), sparse work (a density ablation at fixed $m$), and problem-specific structure
(generic methods against Sinkhorn). Baselines, instance generators, the arithmetic model and the reference optimum are in
Appendix~\ref{app:exp-full}; all curves report the best-so-far relative dual gap
$(\phi(\lambda)-\phi^\star)/|\phi_\se^\star|$ as medians over five runs. The plotted variant,
adaptive ACD, shrinks the conservative gap term using checks on the sampled coordinate, which is not
covered by Theorems~\ref{thm:phase-short}--\ref{thm:work-short}; the fixed-constant schedule is
measured separately.

\begin{table}[t]
\centering\small
\caption{Grouped by step type; GRN \citep{doikov2023qsc} runs on both duals
(Proposition~\ref{prop:se-qsc}). CRN \citep{nesterov2006cubic}, accelerated order-$2$ tensor
\citep{dvurechensky2024near}, APDAGD \citep{dvurechensky2018computational}, RCD
\citep{lobanov2024}. ``$\eps$-dep.'' is the accuracy exponent of the native guarantee (none for the adaptive variant); $\alpha$ is
fitted over $m\in\{100,\dots,1000\}$ at target $10^{-9}$, wall-clock (s) the median over five seeds. A
dash marks a target the median seed missed in budget, or too few points to fit.}
\label{tab:crossover}
\begin{tabular}{@{}llrrrrr@{}}
\toprule
& step & $\eps$-dep. & $\alpha$ (arith.) & $\alpha$ (wall) & \multicolumn{2}{c}{wall-clock} \\
\cmidrule(lr){6-7}
& & & & & $m{=}300$ & $m{=}1000$ \\
\midrule
GRN on $\phi_\lse$ & Newton     & $\log\tfrac1\eps$ & $2.92$ & $2.18$ & $0.016$ & $0.37$ \\
GRN on $\phi_\se$                               & Newton     & $\log\tfrac1\eps$ & $2.89$ & $2.22$ & $0.019$ & $0.43$ \\
CRN & cubic      & $\eps^{-1/2}$     & $3.03$ & $2.13$ & $0.087$ & $1.53$ \\
APDAGD & gradient   & $\eps^{-1/2}$     & $0.92$ & $0.69$ & $0.044$ & $0.16$ \\
Acc. tensor & tensor     & $\eps^{-1/3}$     & ---    & ---    & $0.43$  & --- \\
\addlinespace[2pt]
RCD & coordinate & $\eps^{-1}$       & $0.96$ & $0.94$ & $0.044$ & $0.14$ \\
RCD, exact coordinate step                      & coordinate & $\eps^{-1}$       & $0.96$ & $0.95$ & $0.070$ & $0.22$ \\
ACD, analysed schedule (ours)                   & coordinate & $\eps^{-1/2}$     & ---    & ---    & ---     & --- \\
ACD, adaptive (ours, heuristic)            & coordinate & none         & $0.94$ & $0.92$ & $0.037$ & $0.11$ \\
\bottomrule
\end{tabular}
\end{table}

\paragraph{Synthetic instances.}
On the main sparse instance ($m=300$, $n=1200$, $\gamma=0.05$, five seeds) the adaptive variant
reaches relative dual-gap levels $10^{-6},10^{-9},10^{-12}$ after $80$, $145$ and $200$ epochs.
Two baselines bracket it. Plain RCD, taking the envelope step its own analysis prescribes, needs
$360$, $620$ and $900$; with the same exact coordinate step our variants take, randomized iterative
Bregman projection \citep{benamou2015bregman}, it needs $185$, $320$ and $460$. Against this baseline, which matches the step rule but samples uniformly, the advantage is a factor $2.2$. These epoch counts are nearly flat in $m$ (Table~\ref{tab:crossover}). On
wall-clock at $10^{-9}$ it is $0.037$s against $0.044$s and $0.070$s: the exact step saves epochs at a
matching cost per epoch. We read the comparison on wall-clock rather than on the operation count,
because the per-nonzero charge of Appendix~\ref{app:exp-full} is not calibrated finely enough across
the two methods to resolve a margin of this size. GRN sits on the other side of the crossover
Section~\ref{sec:lse-short} locates: $15$ iterations and $0.016$s here, against $0.37$s at $m=1000$
where the coordinate methods need $0.11$s.

\begin{figure}[t]
\centering
\includegraphics[width=\linewidth]{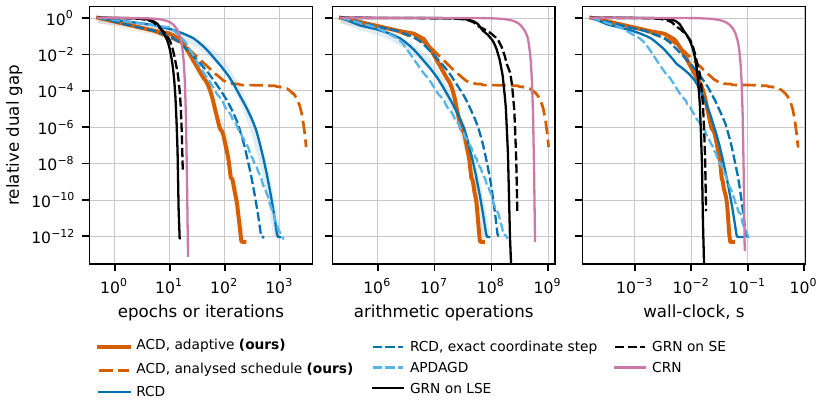}
\caption{Synthetic sparse entropy-regularized LP instance, $m=300$, on the three axes the
experiments separate; medians over five seeds, with min--max bands on the adaptive and RCD curves.
Two ACD curves are shown: the analysed schedule of Theorem~\ref{thm:phase-short} with its proved
constants (dashed) and the adaptive variant that rescales them (solid);
Table~\ref{tab:cert-sweep} measures the gap between them. Epoch counts are not work-comparable
across the coordinate methods (Appendix~\ref{app:exp-full}).}
\label{fig:synth-short}
\end{figure}

At fixed $m=1000$, as the row density varies over a factor of $36$ in $\nnz(A)$, the epochs to
$10^{-9}$ move only between $125$ and $250$ while the arithmetic grows from $2.9\cdot10^7$ to
$6.3\cdot10^8$, within a factor $1.8$ of proportional to $\nnz(A)$. That is what the $O(d_B)$ block
cost predicts: density changes the price of a coordinate step, not how many steps are useful
(Figure~\ref{fig:density-short}, Appendix~\ref{app:exp-full}).

\paragraph{The analysed schedule.}
Run with its proved constants and the fixed anchor of Theorem~\ref{thm:work-short}, the same method
reaches $10^{-6}$ on the $m=300$ instance in $2735$ epochs but not $10^{-9}$ within $3000$ epochs,
and with the envelope step not even $10^{-6}$. On the $m=1000$ instance of Table~\ref{tab:cert-sweep}
its termination test is met at none of the thirty settings. At the analysed scale the $10^{-12}$
tolerance is reached only at $r=0.5\|\lambda_0-\lambda^\star_\se\|_2$, below the true distance and so
outside Theorem~\ref{thm:phase-short}, where the iterate crosses $10^{-6}$ after $1720$ epochs
against $345$ for RCD and $175$ for its exact-step version; at admissible radii it still crosses
$10^{-6}$, in $205$ epochs at $r=20$, but the schedule cannot recognize the crossing.

Rescaling only at phase boundaries recovers most of the gap: with phases capped at $100$ epochs and
$\kappa$ rescaled once per boundary by the slack the phase measured, the constants are frozen within
each phase, so Lemma~\ref{lem:presampling} and Section~\ref{sec:lazy} hold verbatim. This variant
reaches $10^{-9}$ in $245$ epochs at $m=300$ and $260$ at $m=1000$, against $320$ for RCD with the
exact step, with no envelope violation; the truncation and the retrospective choice of $\kappa$
remain outside the theorems.

Instrumenting the sampled coordinates, the envelope was never violated over thirty settings and
five seeds, but it exceeds the curvature actually encountered by $13.6$ at $r=0.5$ and $2.9\cdot10^4$
at $r=20$: the inequality is sound, while the schedule built on it is not competitive at its analysed
scale (Section~\ref{sec:limits}).

\paragraph{Optimal transport.}
The DOTmark experiment \citep{schrieber2017dotmark} ($32\times32$ images, $m=2048$, $\gamma=0.01$)
is unfavorable to a generic method and the margin is large. Sinkhorn
\citep{cuturi2013sinkhorn} reaches $10^{-6}$ on all five pairs in $10$ to $60$ sweeps, the adaptive
variant on three (in $1906$, $2390$ and $2850$ epochs) and the analysed schedule and RCD on none;
where both finish the factor is $32$ to $71$ on epochs and $230$ to $494$ on wall-clock. GRN, CRN and
APDAGD finish on all five, so here the non-separable dual is the better route as well. The gap is
algorithmic rather than one of data touched: Sinkhorn minimizes exactly over a $1024$-dimensional
marginal block, and accelerated \citep{xu2026accsinkhorn} and greedy \citep{altschuler2017near}
variants are faster still. That half-step is exact block minimization, which our analysis admits; instantiated
there, the analysed schedule reaches $10^{-6}$ on three pairs in $45$, $130$ and $750$ block passes
and within a factor $3.5$ on the other two, yet remains $1.9$ to $22.9$ half-steps behind Sinkhorn,
consistent with $S_0$, which moves by $1.2$ to $3.8$, never in our favour
(Appendices~\ref{app:cert-ot} and~\ref{app:exp-full}).

\section{Discussion and limitations}\label{sec:limits}
On the separable side our analysis is much too loose to account for the win of the adaptive
variant, which carries no guarantee, and three things limit what has been shown.

\emph{Slack.} The envelope is three to four orders of magnitude loose at its analysed scale
(Section~\ref{sec:exp-short}), which is why the plotted variant rescales it. Its coefficient on the
\emph{global} gap is not improvable (Proposition~\ref{prop:nu-tight}); the remaining aggregate factor
$\sum_i\rho_i/(\sqrt m\,\rho)$, five to twelve here, is exactly the price of fixing the constants
before the draw (Proposition~\ref{prop:precommit}), so closing it needs a scheme that is not
pre-committed per row.

\emph{Inputs.} The radius $r$, the gap bound $\bar\Delta$ and the partition $\mathcal B$ are
supplied; too small an $r$ or $\bar\Delta$ voids Theorem~\ref{thm:phase-short} undetectably, and our
computable radius is loose, so the runs use an oracle radius. A verified-doubling
radius in the spirit of \citet{lobanov2026} and a budget-driven sampling law are the decisive open
items; for $\mathcal B$ we give no selection rule, and the only partition we test beyond singletons is
the marginal partition of transport.

\emph{Numerical range.} Unlike $\phi_\lse$, $\phi_\se$ is not invariant under a common logit
offset, so log-sum-exp recentering requires tracking the removed scale explicitly; our
implementation does not, which puts small-$\gamma$ transport outside double precision without
\citet{schmitzer2019stabilized}; our theorems assume exact arithmetic, and
the $\sqrt{s_{\rm col}}$ and polynomial separations compare upper bounds under different accounting, not
runtimes (Appendices~\ref{app:l1-target} and~\ref{app:limits-full}).
\label{body:lastcontent}

\subsection*{AI use statement}
Large language models were used in the following parts of this work. \emph{Writing:} editing
and polishing the text and restructuring the \LaTeX{} source. \emph{Mathematical claims:} deriving and checking proofs of the stated results. \emph{Research execution:} implementing the methods
and baselines, designing and running the experiments, preparing the figures, and assisting in the
interpretation of the results. \emph{Synthetic data:} writing the generators of the synthetic
instances described in Appendix~\ref{app:exp-full}; no data were generated by a model directly.
The authors reviewed and verified all statements, proofs, code, and reported results, and take full
responsibility for the content of the paper.

\subsection*{Reproducibility statement}
Every formal result is stated with its assumptions, and its complete proof is given in the
appendices, together with the sparse update and the LSE formulas used by the methods. We do not
include code with this submission; instead, Appendix~\ref{app:exp-full} specifies what is needed to
reimplement the experiments: the synthetic instance generator with its distributions and parameters,
the DOTmark classes, image pairs, cost normalization and regularization, the seeds, the common
starting point and the relative-gap metric, the reference solutions, the inputs supplied to each
method (including the oracle radius used by the analysed schedule and the rescaling rule of the
adaptive variant), and the arithmetic cost model. Wall-clock times were measured on a single laptop
with single-threaded BLAS and are indicative only. Where a reported number aggregates over seeds or
image pairs, the text states which aggregate it is.
\label{body:laststatement}

\clearpage
\bibliographystyle{plainnat}
\bibliography{references_iclr2027}

\appendix

\section{Dual derivations and the gap identities}\label{app:duals}

\begin{assumption}[Mass and feasibility]\label{ass:primal}
Some $q\in\R^m$ satisfies $A^\top q=\mathbf1$ and $\ip{q}{b}=1$, and the affine system admits a
strictly positive point: $A\bar x=b$ for some $\bar x>0$.
\end{assumption}

\noindent The first half forces $\mathbf1^\top x=1$ whenever $Ax=b$, so the feasible set is a
compact subset of the simplex; the second is what makes the minimizer positive in the Attainment
paragraph below. Throughout, the unnormalized and normalized responses are
\begin{equation}\label{eq:responses-short}
x_j(\lambda):=\exp\!\Bigl(\frac{a_j^\top\lambda-c_j}{\gamma}\Bigr),\qquad
Z(\lambda):=\sum_{j=1}^nx_j(\lambda),\qquad
p_j(\lambda):=\frac{x_j(\lambda)}{Z(\lambda)} .
\end{equation}
The SE dual follows by minimizing the Lagrangian over $x\ge0$ coordinatewise; the LSE dual follows
by retaining the simplex constraint and minimizing over it. Under $A^\top q=\mathbf1$ and
$\ip{q}{b}=1$,
\[
\phi_\se(\lambda+tq)=-\ip{b}{\lambda}-t+\gamma e^{t/\gamma}Z(\lambda),
\]
whose stationary point in $t$ proves Proposition~\ref{prop:dual-short}.

\paragraph{Attainment.}
The feasible set is a closed subset of the simplex, hence compact, and strict feasibility supplies
a point with all components positive. The entropy-regularized minimizer $x^\star$ is itself
positive: moving from a point with a zero component toward a strictly feasible point produces a
one-sided entropy derivative of $-\infty$ at that component, while the derivatives at positive
components and of the linear term are finite. The equality-constrained KKT conditions then give a
finite $\lambda^\star_\se$ with
\[
A^\top\lambda^\star_\se=c+\gamma\log x^\star,\qquad Ax^\star=b .
\]
Both dual optima are attained, though dual minimizers need not be unique.

\paragraph{The two gaps.}
Proposition~\ref{prop:dual-short} gives $\phi_\se^\star=\phi_\lse^\star+\gamma$, and subtracting
the two dual functions gives $\phi_\se-\phi_\lse=\gamma(Z-\log Z)$, which is \eqref{eq:gap-split}.
For \eqref{eq:lse-kl}, since $\mathbf1^\top x^\star=1$,
\[
\gamma\KL(x^\star\|p(\lambda))
=f_s(x^\star)-\ip{b}{\lambda}+\gamma\log Z(\lambda)
=f_s^\star+\phi_\lse(\lambda)=F_\lse(\lambda),
\]
where $f_s(x)=\ip{c}{x}+\gamma\sum_jx_j\log x_j$ is the shifted primal objective, equal to the
objective of \eqref{eq:primal-short} plus $\gamma$. The raw SE gap therefore carries a
mass-normalization error on top of the error of the normalized response.

\paragraph{Moments of the two duals.}
The same distinction organizes the derivatives. Let $X$ take the value $a_j$ with probability
$p_j(\lambda)$ and put $\mu=\E[X]=Ap(\lambda)$; then
\begin{equation}\label{eq:both-moments}
\nabla\phi_\se=Z\mu-b,\quad
\nabla^2\phi_\se=\frac{Z}{\gamma}\E[XX^\top],
\qquad
\nabla\phi_\lse=\mu-b,\quad
\nabla^2\phi_\lse=\frac1\gamma\Cov(X),
\end{equation}
an unnormalized second moment against a centered one. Everything that separates the two routes
follows from that one difference, as Table~\ref{tab:dual-comparison} summarizes.

\begin{table}[t]
\centering\small
\caption{Two exact dual representations of the same primal problem. The last two rows are what
Sections~\ref{sec:alg-short} and~\ref{sec:lse-short} spend the paper measuring.}
\label{tab:dual-comparison}
\begin{tabular}{@{}lll@{}}
\toprule
& Sum-exp (SE) & Log-sum-exp (LSE) \\
\midrule
Primal response   & unnormalized $x_j(\lambda)$   & normalized $p_j(\lambda)$ \\
Nonlinear term    & sum of exponentials           & log of a sum of exponentials \\
Hessian           & weighted second moment        & covariance \\
Coupling          & partially separable           & global normalization \\
Global curvature  & unbounded                     & bounded by the columns of $A$ \\
Natural update    & sparse randomized blocks      & regularized Newton steps \\
\bottomrule
\end{tabular}
\end{table}

\section{Proof of the Hessian--gap theorem}\label{app:hg-proof}
We first isolate the scalar inequality. For every $\theta>1$,
\begin{equation}\label{eq:scalar-entropy-app}
e^t\le\theta\,h(t)+c_\theta,\qquad h(t)=e^t-1-t,\qquad c_\theta=\theta\log\frac{\theta}{\theta-1}.
\end{equation}
Indeed $(1-\theta)e^t+\theta(1+t)$ is concave and attains its maximum at $e^t=\theta/(\theta-1)$,
where its value is $c_\theta$.

Let $\lambda^\star$ be an SE minimizer and put $t_j=a_j^\top(\lambda-\lambda^\star)/\gamma$.
Stationarity gives $x_j(\lambda)=x_j^\star e^{t_j}$ and $Ax^\star=b$, so expanding the dual
objective around $\lambda^\star$ yields the exact Bregman identity
\begin{equation}\label{eq:entropy-gap-app}
F(\lambda)=\gamma\sum_jx_j^\star h(t_j).
\end{equation}
For any $u\in\R^{|B|}$,
\begin{align*}
u^\top H_{BB}(\lambda)u
&=\frac1\gamma\sum_jx_j^\star e^{t_j}(A_{B,j}^\top u)^2\\
&\le\frac{c_\theta}{\gamma}\sum_jx_j^\star(A_{B,j}^\top u)^2
 +\frac{\theta}{\gamma}\sum_jx_j^\star h(t_j)(A_{B,j}^\top u)^2\\
&\le c_\theta\,u^\top H^\star_{BB}u+\frac{\theta\rho_B^2}{\gamma^2}F(\lambda)\|u\|_2^2,
\end{align*}
using \eqref{eq:scalar-entropy-app}, then $\|A_{B,j}\|_2\le\rho_B$ and
\eqref{eq:entropy-gap-app}. Since $u$ was arbitrary this proves \eqref{eq:hg-short}, and
$\theta=2$ gives $c_2=2\log2$.

For the universal bound $\bar H_B=\rho_B^2/\gamma$,
$\|H^\star_{BB}\|_2\le\gamma^{-1}\sum_jx_j^\star\|A_{B,j}\|_2^2\le\rho_B^2/\gamma$ because
$\sum_jx_j^\star=1$. For the nonnegative overlap term, Cauchy--Schwarz over the nonzeros of
each block column gives $(A_{B,j}^\top u)^2\le\omega_{Bj}\sum_{i\in B}A_{ij}^2u_i^2$, hence
\[
u^\top H^\star_{BB}u
\le\frac1\gamma\sum_{i\in B}u_i^2\sum_jx_j^\star\omega_{Bj}A_{ij}^2
\le\frac1\gamma\max_{i\in B}\Bigl[b_i\max_{j:A_{ij}\ne0}(\omega_{Bj}A_{ij})\Bigr]\|u\|_2^2,
\]
where we used $A\ge0$ and $\sum_jA_{ij}x_j^\star=b_i$. Both terms bound $\|H^\star_{BB}\|_2$, so their
minimum does, which is \eqref{eq:h-overlap-short}; the singleton formula follows from
$\omega_{\{i\}j}=1$, and the minimum is needed because for $|B|>1$ the overlap term can exceed
$\rho_B^2$ by up to $\omega_{\max}$.

\paragraph{Proof of Corollary~\ref{cor:range-moment}.}\label{proof:range-moment}
For $a\in[\alpha_i^-,\alpha_i^+]$ we have $(a-\alpha_i^-)(a-\alpha_i^+)\le0$, so
$a^2\le(\alpha_i^-+\alpha_i^+)a-\alpha_i^-\alpha_i^+$. Multiplying by $x_j^\star$ and summing at
$a=A_{ij}$ gives $\sum_jx_j^\star A_{ij}^2\le\mu_i$. If $\alpha_i^-<\alpha_i^+$, put
$t=(b_i-\alpha_i^-)/(\alpha_i^+-\alpha_i^-)\in[0,1]$; then
$\mu_i=(1-t)(\alpha_i^-)^2+t(\alpha_i^+)^2$, so
$0\le\mu_i\le\max\{(\alpha_i^-)^2,(\alpha_i^+)^2\}=\|A_{i:}\|_\infty^2$. For a constant row the
same follows from $b_i=\alpha_i^-=\alpha_i^+$ and $\mu_i=(\alpha_i^-)^2$. If $\alpha_i^-=0$ the
numerator is $\alpha_i^+b_i$.

\subsection{Optimality of the growth coefficient}\label{app:nu-tight}
Theorem~\ref{thm:hg-short} bounds a curvature that grows without limit by a constant
$\nu_B=\theta\rho_B^2/\gamma^2$ built from a single entry of the data, which invites the
suspicion that the step $A_{B,j}^\top u\le\rho_B$ is wasteful. It is not.

\begin{proposition}[$\nu_B$ is optimal up to $\theta$]\label{prop:nu-tight}
Let $\lambda^\star$ be an SE minimizer and let $B$ satisfy $A_{B,:}\ne0$. If some positive
semidefinite $\Psi_B$ and some $\nu_B\ge0$ satisfy
\begin{equation}\label{eq:nu-tight-hyp}
\nabla^2_{BB}\phi_\se(\lambda)\preceq\Psi_B+\nu_BF(\lambda)I_B
\qquad\text{for every }\lambda\in\R^m,
\end{equation}
then $\nu_B\ge\rho_B^2/\gamma^2$. Moreover, $\theta=1$ cannot be reached in general: for a
scaled indicator row, $\nu_i=\rho_i^2/\gamma^2$ admits no finite $\kappa_i$.
\end{proposition}
\noindent The second statement is proved only for indicator rows, and we do not characterize
which rows force $\theta>1$; the first statement holds for every block with $A_{B,:}\ne0$.
\begin{proof}
Choose $j_0\in\arg\max_j\|A_{B,j}\|_2$ and set $u:=A_{B,j_0}/\rho_B$, a unit vector. Along
$\lambda_\tau:=\lambda^\star+\gamma\tau U_Bu$ we have $t_j(\tau)=\tau\beta_j$ with
$\beta_j:=A_{B,j}^\top u$, so Cauchy--Schwarz gives $\beta_j\le\|A_{B,j}\|_2\le\rho_B$ while
$\beta_{j_0}=\rho_B$; thus $\max_j\beta_j=\rho_B$ exactly. Put
$J^+:=\{j:\beta_j=\rho_B\}\ni j_0$ and $x^+:=\sum_{J^+}x_j^\star\ge x^\star_{j_0}>0$.
Contracting \eqref{eq:nu-tight-hyp} with $u$,
\[
\frac1\gamma\sum_jx_j^\star e^{\tau\beta_j}\beta_j^2
\le u^\top\Psi_Bu+\nu_B\,\gamma\sum_jx_j^\star h(\tau\beta_j).
\]
Multiply by $e^{-\tau\rho_B}$ and let $\tau\to+\infty$. Every $\beta_j\le\rho_B$ and the sums are
finite, so the left side tends to $x^+\rho_B^2/\gamma$, while $u^\top\Psi_Bu\,e^{-\tau\rho_B}\to0$
and $e^{-\tau\rho_B}\sum_jx_j^\star h(\tau\beta_j)\to x^+$, the $-1-\tau\beta_j$ parts vanishing.
Hence $x^+\rho_B^2/\gamma\le\nu_B\gamma x^+$, which is the claim.

For the second statement take $A_{ij}\in\{0,\rho_i\}$, so that $b_i=\rho_ix^+_i$ with
$x^+_i=\sum_{A_{ij}=\rho_i}x^\star_j$. Along $\lambda_\tau=\lambda^\star+\gamma\tau e_i$ one has
$t_j=\tau A_{ij}$, hence $\partial^2_{ii}\phi_\se=(\rho_i^2x^+_i/\gamma)e^{\tau\rho_i}$ and
$F=\gamma x^+_i(e^{\tau\rho_i}-1-\tau\rho_i)$, so
\[
\partial^2_{ii}\phi_\se(\lambda_\tau)-\frac{\rho_i^2}{\gamma^2}F(\lambda_\tau)
=\frac{\rho_ib_i}{\gamma}(1+\tau\rho_i)\xrightarrow[\tau\to\infty]{}+\infty .
\]
\end{proof}

So the affine envelope of Theorem~\ref{thm:hg-short} is optimal in its class up to the factor
$\theta$, and the exact envelope
$\sup\{\partial^2_{ii}\phi_\se:F\le\Delta\}=(\rho_i^2/\gamma^2)\Delta+O(\log\Delta)$ is affine only
asymptotically, the correction being $H^\star_{ii}\log\Delta+O(1)$ for scaled indicator rows but
$O(1)$ or even negative for signed ones; $\theta$ is the price of the affineness that
Algorithm~\ref{alg:main-short} needs in order to prebuild its sampling tables. One consequence
for Section~\ref{sec:geometry}: $d_A=\sum_i\rho_i^2/\rho^2$ is the far-field dimension of the
\emph{per-row} envelopes. Proposition~\ref{prop:nu-tight} does not make the aggregate
$\sum_i\nu_i$ unimprovable, and Appendix~\ref{app:local-gap} shows it is loose by
$\sum_i\rho_i/(\sqrt m\,\rho)$.

\subsection{Localized gaps and the aggregate slack}\label{app:local-gap}
Proposition~\ref{prop:nu-tight} closes off any \emph{per-row} improvement. The aggregate is a different matter, and sparsity is
the reason: the curvature of row $i$ involves only the columns in $\supp(A_{i:})$, whereas $F$ sums
over all $n$ of them.

\begin{proposition}[Localized gap and its budget]\label{prop:local-gap}
For each $i$ with $\rho_i>0$ define the \emph{localized gap}
\begin{equation}\label{eq:local-gap-def}
F_i(\lambda):=\frac{\gamma}{\rho_i^2}\sum_jx_j^\star h(t_j)A_{ij}^2,
\qquad\text{so}\qquad 0\le F_i(\lambda)\le F(\lambda).
\end{equation}
Then for every $\theta>1$ and every $\lambda$,
\begin{equation}\label{eq:local-gap-bound}
\partial^2_{ii}\phi_\se(\lambda)\le c_\theta H^\star_{ii}+\frac{\theta\rho_i^2}{\gamma^2}F_i(\lambda),
\end{equation}
and the localized gaps obey the budget identity, followed by the bound it implies,
\begin{equation}\label{eq:local-gap-budget}
\sum_i\rho_i^2F_i(\lambda)=\gamma\sum_jx_j^\star h(t_j)\|a_j\|_2^2\le\rho^2F(\lambda).
\end{equation}
\end{proposition}
\begin{proof}
For \eqref{eq:local-gap-bound}, repeat the proof of Theorem~\ref{thm:hg-short} for the singleton
block $\{i\}$ and stop before replacing $A_{ij}^2$ by $\rho_i^2$; the surviving sum is exactly
$\rho_i^2F_i/\gamma$. For \eqref{eq:local-gap-budget}, exchange the two sums:
$\sum_i\rho_i^2F_i=\gamma\sum_jx_j^\star h(t_j)\sum_iA_{ij}^2$, and
$\sum_iA_{ij}^2=\|a_j\|_2^2\le\rho^2$, while $\gamma\sum_jx^\star_jh(t_j)=F$.
\end{proof}

The consequence for the quantity the algorithm is actually charged follows from a
Cauchy--Schwarz step over \eqref{eq:local-gap-budget}: with $\nu_i=\theta\rho_i^2/\gamma^2$,
\begin{equation}\label{eq:aggregate-growth}
\sum_i\sqrt{\nu_iF_i(\lambda)}
=\frac{\sqrt\theta}{\gamma}\sum_i\rho_i\sqrt{F_i}
\le\frac{\sqrt{\theta m}}{\gamma}\,\rho\sqrt{F},
\qquad\text{versus}\qquad
\sum_i\sqrt{\nu_iF}=\frac{\sqrt\theta}{\gamma}\sqrt F\sum_i\rho_i .
\end{equation}
The ratio is $\sqrt m\,\rho/\sum_i\rho_i$. Here $\rho$ is a column $2$-norm and $\rho_i$ a
row $\infty$-norm, so with comparable rows this is $\sqrt{s_{\rm col}/m}$ rather than $1/\sqrt m$, where
$s_{\rm col}$ is the column sparsity. On the synthetic family of Appendix~\ref{app:exp-full} ($s_{\rm col}=19$) it is
$0.196$ at $m=300$, $0.127$ at $m=1000$ and $0.083$ at $m=2000$ -- medians over the five seeds of the
generator, the spread at $m=1000$ being $0.106$ to $0.129$ -- a factor of five to twelve in the
growth channel. The prediction $\sqrt{s/m}$ gives $0.25$, $0.14$ and $0.098$, which the measured
values undershoot by $22$, $8$ and $15$ percent: the same order, not a fixed offset. This is the
channel that dominates the early phases.

\subsection{The shared-budget form and the cost of pre-commitment}\label{app:budget}
Proposition~\ref{prop:local-gap} is best read as a restatement of the BGG condition in which the
rows share one growth budget instead of each owning a coefficient. Writing
$w_i:=\nu_iF_i(\lambda)$ and $\nu_\bullet:=\theta\rho^2/\gamma^2$, \eqref{eq:local-gap-bound} and
\eqref{eq:local-gap-budget} say exactly
\begin{equation}\label{eq:shared-budget}
\partial^2_{ii}\phi_\se(\lambda)\le\kappa_i+w_i,
\qquad 0\le w_i\le\nu_iF(\lambda),
\qquad \sum_iw_i\le\nu_\bullet F(\lambda),
\end{equation}
the per-row cap being $F_i\le F$, which holds because $A_{ij}^2\le\rho_i^2$;
and $\sum_i\nu_i=\nu_\bullet d_A$, so the shared budget is smaller than the sum of the individual
coefficients by exactly the far-field effective dimension. One consequence is immediate and free:
combining \eqref{eq:shared-budget} with the trace form \eqref{eq:global-hess-trace} gives
\begin{equation}\label{eq:global-hess-budget}
\|\nabla^2\phi_\se(\lambda)\|_2\le\sum_i\kappa_i+\nu_\bullet F(\lambda),
\end{equation}
which on the synthetic family of Appendix~\ref{app:exp-full} is a factor $7.8\cdot10^3$ below
$C(\Delta)^2:=(\sum_B\sqrt{\kappa_B+\nu_B\Delta})^2$, the quantity the sampling normalization
forces on the schedule (Appendix~\ref{app:acc-proof}), at $m=300$ and the $k=0$ gap level, of which
$26.6$ comes from the shared budget
(close to $d_A=26.9$, which it approaches as $\Delta$ grows) and the remaining $292$ from
\eqref{eq:global-hess-trace}.

The rate, however, does not follow, and the reason is worth stating as a limitation of the design
rather than of the analysis.

\begin{proposition}[Weight admissible under the cancellation identity]\label{prop:precommit}
Fix a gap level $\Delta$ and write $B_\star=\nu_\bullet\Delta$ for the growth budget in
\eqref{eq:shared-budget} (the symbol $W$ is reserved for arithmetic work). Consider any scheme of
the form of Algorithm~\ref{alg:main-short} whose sampling law $p$ and weight $a_{k+1}$ are fixed
before the block is drawn and which satisfies the cancellation
$p_i\ge a_{k+1}\sqrt{L_i/T_{k+1}}$ for \emph{every} profile admissible under
\eqref{eq:shared-budget}, that is every $w$ with $0\le w_i\le\nu_i\Delta$ and
$\sum_iw_i\le B_\star$. Then
\begin{equation}\label{eq:precommit-bound}
a_{k+1}\le\frac{\sqrt{T_{k+1}}}{\sum_i\sqrt{\kappa_i+\nu_i\Delta}},
\end{equation}
while a profile-aware normalization may use $\max_w\sum_i\sqrt{\kappa_i+w_i}$ over the same set.
The price of pre-commitment, the ratio of the two, satisfies
\begin{equation}\label{eq:precommit-price}
\frac{\sum_i\sqrt{\kappa_i+\nu_i\Delta}}{\sqrt{m\bigl(\sum_i\kappa_i+B_\star\bigr)}}
\ \le\ \text{price}\ \le\ \sqrt{d_A},
\end{equation}
with equality on the left whenever the capped water-filling level is feasible. In the
growth-dominated regime the left-hand side is $\sum_i\rho_i/(\sqrt m\,\rho)$, the reciprocal of the
aggregate ratio of Proposition~\ref{prop:local-gap}.
\end{proposition}
\begin{proof}
Summing $p_i\ge a_{k+1}\sqrt{L_i/T_{k+1}}$ over $i$ with $\sum_ip_i=1$ gives
$a_{k+1}\le\sqrt{T_{k+1}}/\sum_i\sqrt{L_i}$. A pre-committed scheme must respect the worst
admissible profile; under the cap that profile is $w_i=\nu_i\Delta$ for every $i$, which gives
\eqref{eq:precommit-bound}. For the upper bound in \eqref{eq:precommit-price}, Cauchy--Schwarz over
the capped set gives $\max_w\sum_i\sqrt{\kappa_i+w_i}\ge\sum_i\sqrt{\kappa_i+\nu_i\Delta/d_A}$,
obtained by the feasible choice $w_i=\nu_i\Delta/d_A$, and the ratio is then at most $\sqrt{d_A}$;
for the lower bound, $\max_w\sum_i\sqrt{\kappa_i+w_i}\le\sqrt{m(\sum_i\kappa_i+B_\star)}$ by
Cauchy--Schwarz, and this sits in the denominator.
\end{proof}

Measured on the synthetic family at the level $B_\star/\sum_i\kappa_i=85$, medians over five seeds,
the price of \eqref{eq:precommit-price} is $5.1$ at $m=300$, $7.9$ at $m=1000$ and $11.9$ at
$m=2000$, against $\sqrt{d_A}=5.2$, $8.1$ and $12.2$: the capped water-filling is feasible here, so
the lower bound is attained. Dropping the cap would instead give $17.2$, $31.4$ and $44.5$, close to
$\sqrt m$; that profile assigns the whole budget to one row, which requires $\rho_i=\rho$ and is not
realizable when the columns are sparse.

The slack identified above is therefore not recoverable by a cleverer choice of per-row constants,
and the price is exactly the aggregate slack of Proposition~\ref{prop:local-gap} read backwards:
what pre-commitment costs is what the localized gap would have saved. Nor have we found a scheme that escapes the price. Committing to the aggregate and
verifying per-row admissibility after the draw is the natural candidate, but the deficit incurred
on a violating row is weighted by $1/\sqrt{\kappa_i+w_i}$, so an adversarial allocation targets
exactly the under-provisioned rows, and the potential of Appendix~\ref{app:acc-proof} carries no
slack to absorb it.

The price is not constant: it tends to $1$ as $mB_\star/\sum_i\kappa_i\to0$, that is, once
the baseline curvature dominates the growth channel \emph{per row}, a threshold a factor $m$ below
the one a careless reading suggests. We tested whether this regime is reachable by raising $\gamma$
at $m=300$, and it is not: $B_\star/\sum_i\kappa_i$ falls only from $85$ to $5.6$ as $\gamma$ goes
from $0.05$ to $1$, still three orders of magnitude above the $1/m$ that would be needed; the
predicted price moves only from $5.1$ to $4.8$ against $\sqrt{d_A}=5.2$, and the measured ratio of
epochs to RCD \emph{widens} from $5.0$ to $9.7$, because the non-accelerated method benefits more
from the easier problem. We therefore do not claim a regime in which the pre-committed schedule is competitive. On the
instances we tested the growth channel dominates $M_k$ at the gap levels the schedule evaluates,
which is also why rescaling $\kappa$ gains less than one power of the rescaling factor: the
$\kappa$-sweep of Table~\ref{tab:cert-sweep-full} moves by roughly $\kappa^{1/4}$ per decade, not
$\kappa^{1/2}$, because $C_1$ floors the schedule and is untouched by that rescaling.

\section{Accelerated phase and work proof}\label{app:acc-proof}

\paragraph{Sampling law versus step length.}
Line~10 asks only for a point no worse than the explicit step, not for the explicit point itself:
$\bar L_{k,B}$ must be an \emph{a priori} bound where it sets the sampling law and the weight
$a_{k+1}$, but where it sets the step length any such point will do, so exact block minimization is covered and
an inexact iterate is admissible once compared with the explicit step. The restriction
$u_{k+1}\in y_k+\operatorname{range}U_{B_k}$ keeps \eqref{eq:lazy-updates} sparse, which is why
Theorem~\ref{thm:work-short} needs it and not only the rate argument.

Throughout, $F=f-f^\star$ and $0\preceq\nabla^2_{BB}f(u)\preceq(\kappa_B+\nu_BF(u))I_B$.

\paragraph{The degenerate case $C_0=0$.}
If $C_0=0$ then every $\kappa_B=0$. For any PSD $H$ and any block vector $d$,
\begin{equation}\label{eq:block-cs-app}
d^\top Hd\le\Bigl(\sum_B\|d_B\|_2\sqrt{\|H_{BB}\|_2}\Bigr)^2,
\end{equation}
by factoring $H=V^\top V$ and applying the triangle inequality to $Vd=\sum_BVU_Bd_B$. Hence
$\|\nabla^2f(u)\|_2\le C_1^2F(u)$. Along any line, $\psi(t):=F(u^\star+td)$ satisfies
$\psi(0)=\psi'(0)=0$ and $\psi''\le C_1^2\|d\|_2^2\psi$, and the integral form of Gronwall's
inequality forces $\psi\equiv0$. The attained problem is therefore constant, which is why the
theorem assumes $C_0>0$.

If a single block has $\kappa_B=\nu_B=0$, its diagonal Hessian block vanishes everywhere; for a
PSD matrix a zero diagonal block forces zero cross blocks, so $\nabla_Bf$ is constant and vanishes
at a minimizer. Dropping that block is exact. The gradient bound below also covers
$\kappa_B+\nu_BF(u)=0$ by applying the descent argument with $M>0$ and letting $M\downarrow0$.

\paragraph{Self-validating block descent.}
Fix $u,B$ and let $M\ge\kappa_B+\nu_BF(u)>0$. Put $g_B=\nabla_Bf(u)$ and, for $g_B\ne0$, move along
$d=-U_Bg_B/\|g_B\|$. As long as the objective has not exceeded $f(u)$, the BGG bound gives
$d^2f(u+\tau d)/d\tau^2\le M$, and a first return to $f(u)$ before $\tau=\|g_B\|/M$ would
contradict the resulting quadratic upper model. Evaluating at that endpoint,
\begin{equation}\label{eq:block-descent-app}
f\Bigl(u-\frac{U_Bg_B}{M}\Bigr)\le f(u)-\frac{\|g_B\|_2^2}{2M},
\end{equation}
and comparing with $f^\star$,
\begin{equation}\label{eq:block-gradient-app}
\|\nabla_Bf(u)\|_2^2\le2F(u)(\kappa_B+\nu_BF(u)).
\end{equation}
Combining \eqref{eq:block-cs-app} with BGG, on the sublevel $\{F\le\Delta\}$,
\begin{equation}\label{eq:global-hess-app}
\|\nabla^2f\|_2\le C(\Delta)^2,\qquad C(\Delta):=\sum_B\sqrt{\kappa_B+\nu_B\Delta}.
\end{equation}
One more application of Cauchy--Schwarz to \eqref{eq:block-cs-app} gives the sharper
\begin{equation}\label{eq:global-hess-trace}
\|\nabla^2f\|_2\le \mathcal T(\Delta):=\sum_B(\kappa_B+\nu_B\Delta)\le C(\Delta)^2,
\end{equation}
since $d^\top Hd\le(\sum_B\|d_B\|_2^2)(\sum_B\|H_{BB}\|_2)$. Writing $\nB$ for the number of blocks,
to keep it apart from the block size $|B|$: for $\nB$ comparable
blocks the two differ by a factor $\nB$, and on the synthetic family of
Appendix~\ref{app:exp-full} the gap is $292$ at $m=300$. The trace form
\eqref{eq:global-hess-trace} is the sharper of the two and is the statement we record; the schedule
below nevertheless uses $C$, because $M_k=\sum_B\sqrt{\bar L_{k,B}}$ is
forced by the sampling normalization and already satisfies $M_k^2\ge\mathcal T(5\Delta_k)$ with room to
spare. Closing that gap therefore requires a different sampling law, not a different
Hessian bound.

\phantomsection\label{proof:presampling}
\paragraph{Safe interpolation and the constant $5$: proof of Lemma~\ref{lem:presampling}.}
Suppose $\Phi=TF(u)+\tfrac12\|v-u^\star\|^2\le\Pi$ and put $\Delta=\Pi/T$. Choose
$M\ge C(5\Delta)$ and $a>0$ with $M^2a^2=T^+:=T+a$, and set $y=(Tu+av)/T^+$. Write
$\bar y=(1-a/T^+)u+(a/T^+)u^\star$ and $d=(a/T^+)(v-u^\star)$, so that $y=\bar y+d$.
Convexity gives $F(\bar y)\le F(u)\le\Delta$. From $M^2a^2=T^+$ we get $a/T^+=1/(M\sqrt{T^+})$, so
the potential bound $\|v-u^\star\|\le\sqrt{2\Pi}$ yields $\|d\|\le\sqrt{2\Delta}/M$. By
summing \eqref{eq:block-gradient-app} over blocks and applying
\eqref{eq:global-hess-trace}, $\|\nabla f(\bar y)\|^2\le2\Delta\,\mathcal T(\Delta)\le2\Delta\,C(\Delta)^2$, so
$\|\nabla f(\bar y)\|\le\sqrt{2\Delta}\,C(\Delta)\le\sqrt{2\Delta}\,M$; only
$M^2\ge\mathcal T(5\Delta)$ is used.

Now suppose the segment $\bar y+\tau d$, $\tau\in[0,1]$, first reached level $5\Delta$. Before that
crossing \eqref{eq:global-hess-app} bounds the Hessian by $M^2$, so Taylor's formula gives
\[
F(\bar y+\tau d)\le\Delta+\tau\sqrt{2\Delta}M\cdot\frac{\sqrt{2\Delta}}{M}
+\frac{\tau^2}{2}M^2\frac{2\Delta}{M^2}
=\Delta+2\tau\Delta+\tau^2\Delta\le4\Delta,
\]
a contradiction. Hence $F(y)\le4\Delta<5\Delta$. Any fixed margin above $4$ would do.

\paragraph{Conditional accelerated cancellation.}
At iteration $k$ set $\Delta_k=\Pi/T_k$, $\sqrt{\bar L_{k,B}}=\sqrt{\kappa_B}+\sqrt{5\nu_B\Delta_k}$
and $M_k=C_0+\sqrt{5\Delta_k}\,C_1$. Then $M_k\ge C(5\Delta_k)$ because
$\sqrt{x+y}\le\sqrt x+\sqrt y$. On the event
$\Phi_k=T_kF(u_k)+\tfrac12\|v_k-u^\star\|^2\le\Pi$, safe interpolation validates every block model
at $y_k$ \emph{before} sampling. The two-table mixture draws exactly
$p_{k,B}=\sqrt{\bar L_{k,B}}/M_k$. Expanding the momentum norm after the sampled update and
combining with \eqref{eq:block-descent-app}, the quadratic terms cancel because
\begin{equation}\label{eq:cancellation-app}
\frac{a_{k+1}^2}{p_{k,B}}=\frac{T_{k+1}p_{k,B}}{\bar L_{k,B}},
\end{equation}
which follows from $M_k^2a_{k+1}^2=T_{k+1}$ and $p_{k,B}=\sqrt{\bar L_{k,B}}/M_k$. In detail, write
$g_B:=\nabla_Bf(y_k)$ and let $B$ be the drawn block, so
$v_{k+1}=v_k-(a_{k+1}/p_{k,B})U_Bg_B$ and
$\tfrac12\|v_{k+1}-u^\star\|^2=\tfrac12\|v_k-u^\star\|^2
-\tfrac{a_{k+1}}{p_{k,B}}\ip{g_B}{(v_k-u^\star)_B}+\tfrac{a_{k+1}^2}{2p_{k,B}^2}\|g_B\|^2$.
Taking the conditional expectation, the middle term becomes
$-a_{k+1}\ip{\nabla f(y_k)}{v_k-u^\star}$ and the last becomes
$\sum_B\tfrac{a_{k+1}^2}{2p_{k,B}}\|g_B\|^2=T_{k+1}\sum_B\tfrac{p_{k,B}}{2\bar L_{k,B}}\|g_B\|^2$ by
\eqref{eq:cancellation-app}, which is exactly $T_{k+1}$ times the decrease
$\E[F(y_k)-F(u_{k+1})\mid\mathcal F_k]$ guaranteed by \eqref{eq:block-descent-app}. What remains is
$T_{k+1}F(y_k)-a_{k+1}\ip{\nabla f(y_k)}{v_k-u^\star}\le T_kF(u_k)$, which is the two convexity
inequalities $F(y_k)+\ip{\nabla f(y_k)}{u_k-y_k}\le F(u_k)$ and
$F(y_k)+\ip{\nabla f(y_k)}{u^\star-y_k}\le0$ weighted by $T_k$ and $a_{k+1}$ and added, using
$T_k(u_k-y_k)+a_{k+1}(v_k-y_k)=0$. Hence
\begin{equation}\label{eq:supermart-app}
\E[\Phi_{k+1}\mid\mathcal F_k]\le\Phi_k\qquad\text{on }\{\Phi_k\le\Pi\}.
\end{equation}

\phantomsection\label{proof:phase-short}
\paragraph{High-probability phase guarantee: proof of Theorem~\ref{thm:phase-short}.}
Stop the process at the first crossing of $\Phi_k>\Pi$. The transition that first crosses is
covered by \eqref{eq:supermart-app}, so the stopped process is a nonnegative supermartingale.
Initially $\Phi_0=T_0F(z)+\tfrac12\|z-u^\star\|^2\le r^2=\delta\Pi$, so Doob's maximal inequality
bounds the probability of any crossing by $\delta$. On the complementary event
$T_NF(u_N)\le\Pi$, and the stopping rule $T_N\ge\Pi/\eps$ gives $F(u_N)\le\eps$.

\paragraph{Deterministic growth of the weights.}
Let $s_k:=\sqrt{T_k}$ (not to be confused with the Bregman exponents $t_j$ of Section~\ref{sec:hg-short}) and $\Gamma:=\sqrt{5\Pi}\,C_1$, so that $M_k=C_0+\Gamma/s_k$. With
$a_{k+1}=s_{k+1}^2-s_k^2$, the defining equation $M_k^2a_{k+1}^2=s_{k+1}^2$ becomes
\begin{equation}\label{eq:weight-growth-app}
s_{k+1}-s_k=\frac{s_{k+1}}{(C_0+\Gamma/s_k)(s_{k+1}+s_k)} .
\end{equation}
Since $1/2\le s_{k+1}/(s_{k+1}+s_k)\le1$,
\[
\frac1{2(C_0+\Gamma/s_k)}\le s_{k+1}-s_k\le\frac1{C_0+\Gamma/s_k} .
\]
While $s_k\le\Gamma/C_0$ the lower bound gives $s_{k+1}\ge s_k(1+1/(4\Gamma))$, so $s_k$ grows
geometrically and reaches $\Gamma/C_0$ after $O((1+\Gamma)\log(\Gamma/(C_0s_0)))$ iterations: for
$\Gamma<1/4$ the factor $1+1/(4\Gamma)$ is large and the count is logarithmic, while for large
$\Gamma$ the warm-up is of order $\Gamma$ up to a logarithm. Thereafter every increment is
at least $1/(4C_0)$. Reaching $s_N\ge\sqrt{\Pi/\eps}$
therefore takes
\[
N=\widetilde O\bigl(1+C_0\sqrt{\Pi/\eps}+\Gamma\bigr)
=\widetilde O\Bigl(1+\frac{rC_0}{\sqrt{\delta\eps}}+\frac{rC_1}{\sqrt\delta}\Bigr),
\]
which is \eqref{eq:query-short}, and the overshoot obeys
\begin{equation}\label{eq:overshoot-app}
s_N\le\sqrt{\Pi/\eps}+1/C_0 .
\end{equation}

\paragraph{From high probability to expectation.}
Let $0<\eps<\bar\Delta$. Every repetition uses the original $z$, the same $r$ and $\bar\Delta$,
failure level $\delta=1/3$, target $\eps/2$, and an independent sampling stream; return the
smallest-objective point of $\{z,u^{(1)},\dots,u^{(J_{\rm rep})}\}$. Each phase succeeds with
probability at least $2/3$, so all fail with probability at most $3^{-J_{\rm rep}}$. Since the
return is safeguarded by $z$,
\[
\E F(\widehat u)\le\eps/2+\bar\Delta\,3^{-J_{\rm rep}}\le\eps,
\qquad J_{\rm rep}=\Bigl\lceil\frac{\log(2\bar\Delta/\eps)}{\log3}\Bigr\rceil .
\]
No success test and no knowledge of $f^\star$ is needed, only objective comparisons. The same
computation gives the high-probability form $\Prob(F(\widehat u)\le\eps/2)\ge1-3^{-J_{\rm rep}}$.
Keeping $z$ fixed is what allows the radius to be reused: a smaller function value at a new anchor
does not imply a smaller distance to $u^\star$.

\begin{table}[htbp]
\centering\small
\caption{The quantities entering the SE complexity bounds, and what each controls.}
\label{tab:constants-main}
\begin{tabular}{@{}lll@{}}
\toprule
quantity & meaning & role \\
\midrule
$\kappa_B$ & baseline of the curvature envelope & stiffness near the solution \\
$\nu_B$ & growth of curvature per unit gap & transient inflation far from it \\
$d_B$ & touched-data cost $|B|+\nnz(A_{B:})$ & arithmetic per sampled block \\
$C_0,C_1$ & unweighted aggregates & number of block queries \\
$S_0,S_1$ & cost-weighted aggregates & total sparse arithmetic \\
\bottomrule
\end{tabular}
\end{table}

\paragraph{Weighted sparse work.}
With $S_0=\sum_Bd_B\sqrt{\kappa_B}$ and $S_1=\sum_Bd_B\sqrt{\nu_B}$, conditional on the history,
\begin{equation}\label{eq:expected-block-cost-app}
\E[d_{B_k}\mid\mathcal F_k]
=\frac{S_0+\sqrt{5\Delta_k}\,S_1}{C_0+\sqrt{5\Delta_k}\,C_1}
=\frac{S_0+(\Gamma/s_k)(S_1/C_1)}{C_0+\Gamma/s_k}
\end{equation}
when $C_1>0$; the first form shows what the identity says, the second is the one the schedule
recursion consumes, and they agree because $\sqrt{5\Delta_k}=\sqrt{5\Pi}/s_k=\Gamma/(s_kC_1)$. This
is the reason for carrying $(C_0,C_1)$ and $(S_0,S_1)$ separately: the first pair controls how many
block queries are made, the second what data those queries read. The lower inequality after \eqref{eq:weight-growth-app} bounds the first contribution
by $2S_0(s_{k+1}-s_k)$, which telescopes to $2S_0 s_N$. For the second, the \emph{lower} inequality after
\eqref{eq:weight-growth-app} reads $1/(C_0s_k+\Gamma)\le2(s_{k+1}-s_k)/s_k=:2x_k$, so
$\sum_k1/(C_0s_k+\Gamma)\le2\sum_kx_k$; with $x_k\le\vartheta_0:=1/(C_0s_0+\Gamma)$ and
$x\le(1+\vartheta_0)\log(1+x)$ on $[0,\vartheta_0]$, and $\sum_k\log(1+x_k)=\log(s_N/s_0)$, the
second contribution sums to $O(\sqrt\Pi\,S_1(1+\vartheta_0)\log(s_N/s_0))$. The dependence on the initial gap sharpens by keeping both
terms in the denominator: for $C_1>0$,
\begin{equation}\label{eq:retained-denominator}
\sqrt\Pi \vartheta_0=\frac{\sqrt\Pi}{C_0s_0+\sqrt{5\Pi}\,C_1}
\le\min\Bigl\{\frac{\sqrt{2\bar\Delta}}{\sqrt\delta\,C_0},\ \frac1{\sqrt5\,C_1}\Bigr\}.
\end{equation}
Together with \eqref{eq:overshoot-app}, at fixed failure level and phase target $\eps/2$,
\[
\E W_{\rm inner}=\widetilde O\Bigl(\frac{S_0}{C_0}+\frac{rS_0}{\sqrt\eps}+rS_1
+S_1\min\bigl\{\tfrac{\sqrt{2\bar\Delta}}{C_0},\tfrac1{C_1}\bigr\}\Bigr).
\]
for the inner work of one phase. Summing over the $J_{\rm rep}$ independent repetitions changes
only logarithmic factors, and the phase-boundary, comparison and reset work is absorbed below, so
the same expression bounds $\E W_{\rm opt}$. If $C_1=0$ then $S_1=0$, the second
table is omitted, and the bound reduces to $\widetilde O(S_0/C_0+rS_0/\sqrt\eps)$ with no division
by $C_1$. Appendix~\ref{app:sparse-alg} absorbs comparisons, candidate storage and resets into the
accumulated touched work; initialization and any requested dense output are separate passes. This
proves Theorem~\ref{thm:work-short}.

\section{Exact sparse implementation and phase boundaries}\label{app:sparse-alg}
The lazy identities were stated in Section~\ref{sec:lazy}; we
verify them and add the safeguard. From $y_k=(T_ku_k+a_{k+1}v_k)/T_{k+1}$ and $T_{k+1}=T_k+a_{k+1}$,
\[
y_k=v_k+\frac{T_k(u_k-v_k)}{T_{k+1}}=z+\tilde v_k+\frac{w_k}{T_{k+1}},
\]
which is the factorization of $y_k$ used there. For the updates, $v_{k+1}=v_k-\xi U_{B_k}g$ gives the first half of
\eqref{eq:lazy-updates} directly, and, writing $u_{k+1}=y_k+\eta_k$ with $\eta_k$ supported on $B_k$,
\[
w_{k+1}=T_{k+1}(u_{k+1}-v_{k+1})
=T_{k+1}\bigl(y_k-v_k+\eta_k+\xi U_{B_k}g\bigr)
=w_k+T_{k+1}\bigl(\xi U_{B_k}g+\eta_k\bigr),
\]
using $T_{k+1}(y_k-v_k)=T_ku_k-T_kv_k=w_k$. The explicit step $\eta_k=-U_{B_k}g/\bar L_{k,B_k}$
recovers the coefficient $\xi-\bar L_{k,B_k}^{-1}$.
The transformed projections follow the same recurrences composed with $A^\top$:
\[
\zeta^v_{k+1}=\zeta^v_k-\xi A^\top U_{B_k}g,\qquad
\zeta^w_{k+1}=\zeta^w_k+T_{k+1}A^\top\bigl(\xi U_{B_k}g+\eta_k\bigr),
\]
and each is supported on $\supp(A_{B_k:})$. With row-oriented storage all four updates and
\eqref{eq:sparse-grad} cost $O(|B_k|+\nnz(A_{B_k:}))=O(d_{B_k})$;
Algorithm~\ref{alg:sparse-step} collects one such step.

\paragraph{Sparse phase safeguards.}
Let $I$ be the union of visited dual blocks and $J=\bigcup_{i\in I}\supp(A_{i:})$. The candidate
displacement $\hat h=\tilde v_N+w_N/T_N$ is supported on $I$ and its transform
$A^\top\hat h=\zeta^v_N+\zeta^w_N/T_N$ on $J$. Storing $e_j^0=\exp((\zeta^0_j-c_j)/\gamma)$ at the phase base,
\begin{equation}\label{eq:sparse-safeguard-app}
\phi_\se(z+\hat h)-\phi_\se(z)
=\gamma\sum_{j\in J}e_j^0\Bigl(e^{(A^\top\hat h)_j/\gamma}-1\Bigr)-\sum_{i\in I}b_i\hat h_i,
\end{equation}
because $(A^\top\hat h)_j=0$ for $j\notin J$. Timestamp arrays or touched-index lists give exact
accept/reject checks and resets in $O(|I|+|J|)$, and $|I|+|J|\le\sum_{k<N}d_{B_k}$, so phase
boundaries are absorbed by the cumulative inner work. No $O(m)$ or $O(n)$ restart scan is hidden.

\paragraph{Fixed-anchor preparation and candidate storage.}
The data $\zeta^0$, $e^0$, the incidence structure, the row and block constants and the
alias tables are prepared once in $O(\nnz(A)+m+n)$ and remain unchanged across the independent
repetitions of Theorem~\ref{thm:work-short}. Each candidate's objective difference from $z$ comes
from \eqref{eq:sparse-safeguard-app}; compare that scalar with the stored best (initialized to
zero) and, if smaller, copy only the current visited-coordinate displacement. Freeing or
overwriting a previous sparse candidate is charged to its earlier copy, so total copying and
resetting is a constant multiple of total visited-set work even when supports differ. Comparisons
of nearly equal values need numerically stable summation; the theorem is stated in exact real
arithmetic and is not a bit-complexity result.

The sampling law is exactly
$p_{k,B}=\tfrac{C_0}{M_k}\pi^{(0)}_B+\bigl(1-\tfrac{C_0}{M_k}\bigr)\pi^{(1)}_B$, the second component
omitted when $C_1=0$, so two static alias tables give $O(1)$ expected sampling after
$O(\nB)$ preprocessing.

\begin{algorithm}[t]
\caption{One exact sparse SE block update}
\label{alg:sparse-step}
\begin{algorithmic}[1]
\State Draw $B_k$ by the two-table mixture of Algorithm~\ref{alg:main-short}.
\State Evaluate $g=\nabla_{B_k}\phi_\se(y_k)$ from incident factors only, via \eqref{eq:sparse-grad}.
\State Update $(\tilde v_k,w_k)$ and $(\zeta^v_k,\zeta^w_k)$ by the four sparse recurrences above.
\State Record touched dual rows and primal factors for sparse phase validation and reset.
\end{algorithmic}
\end{algorithm}

\section{The LSE high-order branch}\label{app:tensor}
\subsection{Gauge, attainment and bounded sublevels}
Let $\mathcal G:=\{h:A^\top h\in\operatorname{span}\{\mathbf1\}\}$ and $\mathcal H:=\mathcal G^\perp$.
If $A^\top h=t\mathbf1$, strict feasibility gives $\ip{b}{h}=t$, so
$\phi_\lse(\lambda+h)=\phi_\lse(\lambda)$; also $A^\top(h-tq)=0$, which proves
$\mathcal G=\ker A^\top+\operatorname{span}\{q\}$. For a simplex point $p$,
$\ip{h}{Ap-b}=t(\mathbf1^\top p-1)=0$, so $Ap-b\in\mathcal H$: an orthonormal parametrization
$\lambda=Q\xi$ with columns of $Q$ spanning $\mathcal H$ has $\|Q^\top(Ap-b)\|_2=\|Ap-b\|_2$, and
the full Euclidean residual is preserved. A nonorthogonal coordinate fixing need not preserve it.

For coercivity on $\mathcal H$, choose strictly feasible $x^{\rm f}>0$ with
$\mathbf1^\top x^{\rm f}=1$ and set $\varsigma(h):=\max_ja_j^\top h-\ip{b}{h}$ for unit
$h\in\mathcal H$. The second term is the average $\sum_jx_j^{\rm f}a_j^\top h$ with strictly
positive weights, so $\varsigma(h)\ge0$ with equality only if all $a_j^\top h$ coincide, i.e.\
$h\in\mathcal G$. If $\dim\mathcal H>0$, compactness of the unit sphere gives
$\varsigma_*:=\min\varsigma>0$, and for $t\ge0$,
$\phi_\lse(th)\ge t\varsigma_*-\max_jc_j$. Hence all sublevels on $\mathcal H$ are compact. The positive
softmax weights make the Hessian strictly positive on nonzero directions of $\mathcal H$, so the
minimizer is unique. If $\mathcal H=\{0\}$ the residual at the sole gauge representative vanishes
and no iterations are needed.

For transport with two complete marginal families, $\lambda=(\alpha,\beta)$ and
$a_{ij}^\top\lambda=\alpha_i+\beta_j$, so
$\mathcal G=\operatorname{span}\{(\mathbf1,0),(0,\mathbf1)\}$: both marginal-constant directions
must be removed, for example by $\mathbf1^\top\alpha=\mathbf1^\top\beta=0$. Removing only one is
insufficient.

\subsection{Derivative bounds and second-order rates}
Section~\ref{sec:lse-short} states the constants \eqref{eq:lse-constants} and the two second-order
branches; this appendix derives them and records the gauge details. Let $X$ take the value $a_j$
with probability $p_j(\lambda)$ and put $\mu=\E[X]$. Differentiating the softmax weights in a
direction $u$ gives $Dp_j(\lambda)[u]=\gamma^{-1}p_j\ip{a_j-\mu}{u}$, so with
$Y_u:=\ip{X-\mu}{u}$,
\begin{equation}\label{eq:lse-moments-app}
\nabla\phi_\lse(\lambda)=\mu-b,\qquad
\nabla^2\phi_\lse(\lambda)=\tfrac1\gamma\Cov(X),\qquad
\nabla^3\phi_\lse(\lambda)[u,v,w]=\tfrac1{\gamma^2}\E[Y_uY_vY_w],
\end{equation}
the two terms from differentiating $\mu$ in the third identity vanishing because
$\E[Y_u]=0$. The same identities survive an orthogonal restriction to $\mathcal H$.

\phantomsection\label{proof:se-qsc}
\paragraph{Proof of Proposition~\ref{prop:se-qsc}.}
Differentiating $\phi_\se$ three times gives
$\nabla^3\phi_\se(\lambda)[u,u,v]=\gamma^{-2}\sum_jx_j(\lambda)(a_j^\top u)^2(a_j^\top v)$, while
$\nabla^2\phi_\se(\lambda)[u,u]=\gamma^{-1}\sum_jx_j(\lambda)(a_j^\top u)^2$. Every term of the
first sum shares the nonnegative weight $x_j(a_j^\top u)^2$ of the second, so pulling out
$\max_j|a_j^\top v|$ and using $|a_j^\top v|\le\rho\|v\|_2$ gives the claim. The bound is attained
in the limit where one column carries the whole response. Note that it needs no sign assumption on
$A$ and no restriction of $\lambda$, unlike the global Hessian bound, which does not exist for
$\phi_\se$.

\phantomsection\label{proof:lse-geometry}
\paragraph{Proof of Theorem~\ref{thm:lse-geometry}.}
Every bound follows from one observation: for a unit $u$ the scalar variable $u^\top X$ is supported
in an interval of length at most $\Delta_A=\max_{j,k}\|a_j-a_k\|_2$, and a variable supported in an
interval of length $d$ has variance at most $d^2/4$. Hence
$u^\top\nabla^2\phi_\lse u=\gamma^{-1}\Var(u^\top X)\le\Delta_A^2/(4\gamma)$, which is the first
bound. For the second, $|Y_v|\le\Delta_A\|v\|_2$ pointwise, so
\eqref{eq:lse-moments-app} gives
$|\nabla^3\phi_\lse[u,u,v]|\le\gamma^{-2}\Delta_A\|v\|_2\E[Y_u^2]
=\gamma^{-1}\Delta_A\|v\|_2(u^\top\nabla^2\phi_\lse u)$, which is quasi-self-concordance with
$M_{\rm qsc}=\Delta_A/\gamma$. For the third, take $u,v,w$ unit, bound $|Y_w|\le\Delta_A$ and apply
Cauchy--Schwarz with the variance bound to get
$\E|Y_uY_v|\le\Delta_A^2/4$; homogeneity extends the result to arbitrary arguments. Since
$\Delta_A\le2\rho$ for $\rho=\max_j\|a_j\|_2=\|A\|_{1\to2}$, the column-norm constants
$L_\lse\le\rho^2/\gamma$, $M_{\rm qsc}\le2\rho/\gamma$ and $M_2\le2\rho^3/\gamma^2$ follow, and the
exponent comparisons below are stated in those. More generally each fixed order $\ell$ has a
derivative Lipschitz bound $O_\ell(\rho^{\ell+1}/\gamma^\ell)$ from the bounded-support cumulant
formula.

\paragraph{Gradient regularization.}
The quasi-self-concordance bound controls relative rather than additive Hessian variation,
$e^{-M_{\rm qsc}\|h\|_2}H(\lambda)\preceq H(\lambda+h)\preceq e^{M_{\rm qsc}\|h\|_2}H(\lambda)$,
which is what licenses a regularizer proportional to the current gradient norm: far from a solution
the shift limits the step, and at feasibility it vanishes. Writing $f=\phi_\lse|_{\mathcal H}$,
$g_k=\nabla f(\lambda_k)$ and $\varrho_k=\|g_k\|_2$ (not the supplied radius $r$ of
Section~\ref{sec:alg-short}), the last quantity is exactly the primal residual
$\|Ap(\lambda_k)-b\|_2$ by \eqref{eq:lse-moments-app}. The constant choice
$\sigma_k=M_{\rm qsc}$ is admissible; if the constant is not known tightly,
\citet{doikov2023qsc} double $\sigma_k$ until
$\ip{g_{k+1}}{-h_k}\ge\varrho_{k+1}^2/(2\sigma_k\varrho_k)$, a test in quantities already computed at the trial
point, which costs at most one extra trial on average with the usual warm start.

\begin{proposition}[GRN on the LSE dual]\label{prop:grn-rate}
Let $\mathcal L_0=\{\lambda\in\mathcal H:f(\lambda)\le f(\lambda_0)\}$ and
$D\ge\sup_{u,v\in\mathcal L_0}\|u-v\|_2$, finite by the coercivity above. If $M_{\rm qsc}>0$ and
$\sigma_k\in(0,2M_{\rm qsc}]$ satisfies the progress test, then for $k\ge1$
\begin{equation}\label{eq:grn-rate}
f(\lambda_k)-f^\star\le e^{-k/(8M_{\rm qsc}D)}\bigl(f(\lambda_0)-f^\star\bigr)+e^{-k/4}\varrho_0D ,
\end{equation}
so an LSE dual gap at most $\eps$ is reached in
$O\bigl((1+M_{\rm qsc}D)\log((f(\lambda_0)-f^\star+\varrho_0D)/\eps)\bigr)$ iterations, with no
strong-convexity assumption. With $M_{\rm qsc}\le\Delta_A/\gamma$ this is the instantiation used in
Section~\ref{sec:lse-short}.
\end{proposition}

\noindent Equation~\eqref{eq:grn-rate} is the Euclidean smooth specialization of
\citet[Theorem~3.3]{doikov2023qsc} applied on $\mathcal H$, whose initial sublevel set has diameter
$D$; Theorem~\ref{thm:lse-geometry} supplies the quasi-self-concordance parameter. The stated
iteration count follows by making each term at most $\eps/2$, which needs
$k\ge8M_{\rm qsc}D\log(2(f(\lambda_0)-f^\star)/\eps)$ and $k\ge4\log(2\varrho_0D/\eps)$ respectively. Near
a solution the local analysis of \citet[Section~5]{doikov2023qsc} adds a quadratic phase in the
corresponding normalized gradient measure. Unlike CRN below, GRN needs no valid Lipschitz-Hessian
constant, but it does need the progress test; CRN needs the constant and no search, at the price of
a cubic subproblem instead of one shifted linear system.

\paragraph{Cubic regularization.}
From $\lambda_0\in\mathcal H$, let $D\ge\sup_{\lambda\in\mathcal L_0}\|\lambda-\lambda^\star_\lse\|_2$
on the initial sublevel $\mathcal L_0$. The cubic step
\begin{equation}\label{eq:crn-model-short}
h_k\in\arg\min_{h\in\mathcal H}\Bigl\{\ip{g_k}{h}+\tfrac12\ip{H_kh}{h}+\tfrac{M_2}{6}\|h\|_2^3\Bigr\}
\end{equation}
is monotone and satisfies $F_\lse(\lambda_k)=O(M_2D^3/k^2)$ \citep{nesterov2006cubic};
Algorithm~\ref{alg:crn} states the iteration. Global
smoothness gives $\|Ap(\lambda_k)-b\|_2^2\le2L_\lse F_\lse(\lambda_k)$, hence
$k_{\rm CRN}(\eps_{\rm p})=O(1+\rho^{5/2}D^{3/2}/(\gamma^{3/2}\eps_{\rm p}))$.

\paragraph{Accelerated primal--dual tensor method, order $2$.}
For zero initial iterates on $\mathcal H$ and $R\ge\|\lambda^\star_\lse\|_2$, apply
\citet[Theorem~5.4]{dvurechensky2024near} to the simplex primal and its LSE dual. With the sign
convention reversed and weights reindexed as $a_{i+1}$, the returned primal point is the weighted
average
\begin{equation}\label{eq:tensor-average-short}
\widehat p_k=\frac1{A_k}\sum_{i=0}^{k-1}a_{i+1}p(\lambda_{i+1}),\qquad A_k=\sum_{i=0}^{k-1}a_{i+1},
\end{equation}
and for a regularization parameter equal to a fixed multiple strictly greater than $M_2$,
\begin{equation}\label{eq:tensor-direct-short}
\|A\widehat p_k-b\|_2=O(M_2R^2/k^3),\qquad
\bigl|f_s(\widehat p_k)+\phi_\lse(\lambda_k)\bigr|\le O(M_2R^3/k^3).
\end{equation}
Hence $k_{\rm accT}(\eps_{\rm p})=O(1+\rho R^{2/3}/(\gamma^{2/3}\eps_{\rm p}^{1/3}))$. This is the
reason the method averages: the estimate is primal--dual and direct, rather than a dual-gap rate
converted to feasibility at the cost of a square root. A diameter or sublevel radius may replace
$R$ only if it actually bounds the distance from the stipulated zero start.

\paragraph{Translation to a nonzero start.}
Let $\lambda_0\in\mathcal H$ and $R\ge\|\lambda^\star_\lse-\lambda_0\|_2$. Writing $\upsilon=\lambda-\lambda_0$,
$\widetilde c=c-A^\top\lambda_0$, $\widetilde f_s(p)=f_s(p)-\ip{\lambda_0}{Ap}$ and
$\widetilde\phi(\upsilon)=\phi_\lse(\lambda_0+\upsilon)+\ip{b}{\lambda_0}$ leaves the primal optimizer and
responses unchanged, and the zero-start theorem gives
\begin{equation}\label{eq:centered-tensor-gap}
\bigl|f_s(\widehat p_k)+\phi_\lse(\lambda_k)-\ip{\lambda_0}{A\widehat p_k-b}\bigr|\le O(M_2R^3/k^3).
\end{equation}
The uncentered gap is bounded only after adding $\|\lambda_0\|_2\|A\widehat p_k-b\|_2$, which is
why the starting point cannot be suppressed in a direct primal--dual guarantee.

\paragraph{Near-optimal gradient-norm method.}
With $R\ge\|\lambda_0-\lambda^\star_\lse\|_2$, the initial-distance specialization of
\citet{dvurechensky2024near} gives
$k_{\rm nearT}(\eps_{\rm p})=\widetilde O(1+(M_2R^2/\eps_{\rm p})^{2/7})
=\widetilde O(1+\rho^{6/7}R^{4/7}/(\gamma^{4/7}\eps_{\rm p}^{2/7}))$.
This is a theoretical benchmark, not an implementation in this paper; its line searches and
regularizations are part of the logarithmic oracle overhead.

\begin{table}[htbp]
\centering\small
\caption{Second-order LSE regimes at the common target $\|Ap-b\|_2\le\eps_{\rm p}$. All methods
start at zero on $\mathcal H$. $D$ is a sublevel radius, $R$ a solution-norm bound. Bounds are understood with an additive constant and exact local subproblem
solves.}
\label{tab:lse-main}
\begin{tabular}{@{}lll@{}}
\toprule
Method & primal output & iterations \\
\midrule
CRN & $p(\lambda_k)$ & $O(\rho^{5/2}D^{3/2}/(\gamma^{3/2}\eps_{\rm p}))$\\
accelerated tensor & $\widehat p_k$ in \eqref{eq:tensor-average-short}
& $O(\rho R^{2/3}/(\gamma^{2/3}\eps_{\rm p}^{1/3}))$\\
near-optimal 2nd order & response at returned iterate
& $\widetilde O(\rho^{6/7}R^{4/7}/(\gamma^{4/7}\eps_{\rm p}^{2/7}))$\\
\bottomrule
\end{tabular}
\end{table}

\begin{algorithm}[t]
\caption{Cubic-regularized Newton on the LSE dual}
\label{alg:crn}
\begin{algorithmic}[1]
\Require $\lambda_0\in\mathcal H$; $\lambda_0=0$ for the comparison in Table~\ref{tab:lse-main}.
\For{$k=0,\ldots,N-1$}
\State $p_k\gets\operatorname{softmax}((A^\top\lambda_k-c)/\gamma)$, $g_k\gets Ap_k-b$
\State $H_k\gets\gamma^{-1}A(\Diag p_k-p_kp_k^\top)A^\top$
\State solve \eqref{eq:crn-model-short} over $h_k\in\mathcal H$; $\lambda_{k+1}\gets\lambda_k+h_k$
\EndFor
\end{algorithmic}
\end{algorithm}

Preparing an explicit orthonormal gauge basis, when used, is additional preprocessing, and the
complexity discussion counts high-order oracle steps and their dense algebra, not unit-cost Hessian
or cubic solves. The quasi-self-concordance constant $M_{\rm qsc}\le2\rho/\gamma$ used by GRN
follows from the same third-derivative formula: bounding one factor by $2\rho\|v\|_2$ and the
remaining quadratic form by $\gamma\,\nabla^2\phi_\lse[u,u]$ gives
$|\nabla^3\phi_\lse[u,u,v]|\le(2\rho\|v\|_2/\gamma)\,\nabla^2\phi_\lse[u,u]$.

\section{Primal reconstruction, certificates and transport}\label{app:cert-ot}
For SE, normalize $\widehat x=x(\lambda)/\|x(\lambda)\|_1$. Writing
$\gKL(u\|v)=\sum_j\bigl(u_j\log(u_j/v_j)-u_j+v_j\bigr)$ for the generalized (unnormalized)
Kullback--Leibler divergence, the exact identity $F_\se(\lambda)=\gamma\gKL(x^\star\|x(\lambda))$
gives
\begin{equation}\label{eq:primal-recovery-short}
\KL(x^\star\|\widehat x)\le F_\se(\lambda)/\gamma,\qquad
\|\widehat x-x^\star\|_1\le\sqrt{2F_\se(\lambda)/\gamma},
\end{equation}
the second by Pinsker's inequality \citep{cover2006elements}. The stronger statement is $F_\lse(\lambda)=\gamma\KL(x^\star\|\widehat x)$ together with
$F_\se=F_\lse+\gamma(Z-1-\log Z)$: the two gaps quantify different errors even at the same
normalized primal point.

\subsection{A radius-based computable certificate}
Suppose $R_{\rm abs}\ge\|\lambda^\star_\lse\|_2$ is supplied for the orthogonal gauge
representative. For any simplex point $w$ and any dual $\lambda$, define
\begin{equation}\label{eq:computable-cert}
\mathcal C_{R_{\rm abs}}(w,\lambda):=f_s(w)+\phi_\lse(\lambda)+R_{\rm abs}\|Aw-b\|_2 .
\end{equation}
Weak duality at $\lambda^\star_\lse$ gives
$f_s(w)-f_s^\star\ge\ip{\lambda^\star_\lse}{Aw-b}\ge-R_{\rm abs}\|Aw-b\|_2$; adding
$F_\lse(\lambda)=\phi_\lse(\lambda)+f_s^\star$ and using $F_\lse\ge0$ yields
\begin{equation}\label{eq:certificate-bounds}
0\le F_\lse(\lambda)\le\mathcal C_{R_{\rm abs}}(w,\lambda),\qquad
-R_{\rm abs}\|Aw-b\|_2\le f_s(w)-f_s^\star\le\mathcal C_{R_{\rm abs}}(w,\lambda).
\end{equation}
Checking $\|Aw-b\|_2\le\eps_{\rm p}$ and
$\max\{\mathcal C_{R_{\rm abs}}(w,\lambda),R_{\rm abs}\|Aw-b\|_2\}\le\eps_{\rm f}$ therefore
certifies both the feasibility tolerance and $|f_s(w)-f_s^\star|\le\eps_{\rm f}$. Taking $w=p(\lambda)$,
the identity $f_s(p(\lambda))+\phi_\lse(\lambda)=\ip{\lambda}{Ap(\lambda)-b}$ makes
\eqref{eq:computable-cert} particularly cheap. This is not a radius-free certificate: a known
$R_{\rm abs}$ is an extra input. The softmax, the residual and the value
\eqref{eq:computable-cert} are formed in one final $O(\nnz(A)+m+n)$ pass using log-sum-exp
evaluation. A known feasible $x^{\rm f}>0$ also
supplies an initial SE gap bound without $f^\star$,
$\bar\Delta=f_s(x^{\rm f})+\phi_\se(z)\ge F_\se(z)$, since $\phi_\se^\star+f_s^\star=0$.

\subsection{Comparison at a common residual tolerance}
Since $\|Av\|_2\le\rho\|v\|_1$ and
$\E\|A\widehat x-b\|_2\le\rho\sqrt{2\E F_\se(\widehat u)/\gamma}$ by Jensen, running
Theorem~\ref{thm:work-short} at $\eps=\gamma\eps_{\rm p}^2/(2\rho^2)$ gives
$\E\|A\widehat x-b\|_2\le\eps_{\rm p}$ at leading accuracy-dependent cost
\begin{equation}\label{eq:se-primal-work}
\widetilde O\!\left(\frac{r\rho}{\gamma\eps_{\rm p}}\sum_Bd_B\sqrt{\gamma\kappa_B}\right).
\end{equation}
For a single phase the same implication holds on its success event. Let
$K_{\rm dense}:=\nnz(A)\omega_{\max}+m^2+m^3+n$ with $\omega_{\max}=\max_j\|a_j\|_0$ denote the
leading dense local algebra plus factor access. Table~\ref{tab:common-primal-work} collects the
resulting leading costs.

\begin{table}[htbp]
\centering\small
\caption{Leading accuracy-dependent arithmetic at residual tolerance $\eps_{\rm p}$. SE also carries
the initialization, transient and output terms of Theorem~\ref{thm:work-short}; dense-method setup
is separate. The radii $r,R,D$ belong to different algorithms and are not equated. No universal
runtime ordering follows from the powers of $\eps_{\rm p}$ alone.}
\label{tab:common-primal-work}
\begin{tabular}{@{}p{0.20\linewidth}p{0.17\linewidth}p{0.55\linewidth}@{}}
\toprule
Method & Guarantee & Leading work \\
\midrule
SE fixed anchor & expectation &
$\widetilde O\!\left(\dfrac{r\rho}{\gamma\eps_{\rm p}}\sum_Bd_B\sqrt{\gamma\kappa_B}\right)$ \\[1.0ex]
LSE CRN & deterministic &
$\widetilde O\!\left(K_{\rm dense}\dfrac{\rho^{5/2}D^{3/2}}{\gamma^{3/2}\eps_{\rm p}}\right)$ \\[1.0ex]
LSE acc.\ tensor & deterministic &
$\widetilde O\!\left(K_{\rm dense}\dfrac{\rho R^{2/3}}{\gamma^{2/3}\eps_{\rm p}^{1/3}}\right)$ \\[1.0ex]
LSE near-optimal & deterministic &
$\widetilde O\!\left(K_{\rm dense}\dfrac{\rho^{6/7}R^{4/7}}{\gamma^{4/7}\eps_{\rm p}^{2/7}}\right)$ \\
\bottomrule
\end{tabular}
\end{table}

\subsection{Transport specialization}
For uniform $d\times d$ transport, singleton and two-marginal block partitions show the difference
between query count and arithmetic. The two-block partition lowers the aggregate curvature query
factor by $d$, but each block costs $d$ times more to touch, so the leading weighted work is
unchanged. This is the structural reason specialized Sinkhorn updates can dominate generic methods
despite favourable block-query counts.

\paragraph{Measured.}
We instantiate Algorithm~\ref{alg:main-short} on this partition, the one run in the paper that is not
singleton. The blocks are the two marginal families, so a column meets exactly one row of each block,
$\rho_B=1$, and \eqref{eq:bgg} holds with $\kappa_B=2\log2\max_{i\in B}b_i/\gamma$ and
$\nu_B=2/\gamma^2$. Exact minimization over a block is the Sinkhorn half-step in the log domain, and
it cannot have a larger value than the explicit step $y-U_B\nabla_Bf(y)/\bar L_B$, so the
admissibility test of Algorithm~\ref{alg:main-short} is met at no cost and the potential argument
applies verbatim. A block pass and a Sinkhorn half-step each read the cost matrix once, so half-steps
are the implementation-independent axis and a Sinkhorn sweep is two of them; the budget is $4000$
passes, what the singleton runs receive in epochs.

Table~\ref{tab:block-partition} reports it. The partition is what makes the analysed schedule usable
on transport at all: with $\kappa=1$ it reaches $10^{-6}$ on three of the five pairs and stops at
$1.9$ to $3.5\cdot10^{-6}$ on the other two, one to two orders below the singleton analysed schedule
taking the same exact step, which stalls between $1.5\cdot10^{-5}$ and $1.5\cdot10^{-4}$ (with the
envelope step instead, between $2.1\cdot10^{-3}$ and $4.3\cdot10^{-2}$). The \emph{rescaled}
singleton variant also reaches $10^{-6}$ on three pairs, in $1906$, $2390$ and $2850$ epochs, but not
the same three; on MicroscopyImages, where both finish, the certified block schedule needs $130$
passes against its $2850$ epochs. What neither does is overtake the method whose step it borrows. Rescaling the gap term by the measured
$F_\se(\lambda_0)/\bar\Delta$ (between $2.1\cdot10^{-2}$ and $2.3\cdot10^{-1}$, and no more covered by
the theorems than the singleton adaptation is) rescues the two pairs that miss but costs
CauchyDensity its own, so neither schedule dominates.

Two checks on the implementation, one of which discriminates and one of which does not. The step the
analysis turns on is the cancellation identity $M_k^2a_{k+1}^2=T_{k+1}$, which is what makes the second
moment of the momentum update cancel against the block descent term. Comparing the two sides along a
trajectory started at $\lambda_0$, they agree to $1.3\cdot10^{-15}$ or better on every pair, and
dropping the $1/p_{k,B}$ from the momentum update moves the disagreement to $7.5\cdot10^{-1}$; that test
has teeth. The potential $\Phi_k=T_kF_\se(u_k)+\tfrac12\|v_k-\lambda^\star_\se\|_2^2$ of
Theorem~\ref{thm:phase-short} also never increases in conditional expectation, taken exactly over both
blocks rather than sampled, over $300$ steps on each pair -- but that test does \emph{not} discriminate,
and the reason is worth reporting: $\bar L_{k,B}$ exceeds $\|\nabla^2_{BB}\phi_\se(y_k)\|_2$ here by a
factor $4.5\cdot10^5$ to $8.0\cdot10^5$, so the inequality has room to absorb a wrong constant. Scaling
$\nu_B$ down by $10^4$ leaves a twelvefold margin and produces no violation. That is the same looseness
Section~\ref{sec:exp-short} measures, seen from the inside. Realized $\Phi_k$ is free to rise; the
statement is a supermartingale, not a monotone decrease. The reference point must be the SE minimizer, which is the
LSE one shifted by $-\gamma\log Z$ along the gauge direction $q$ (Proposition~\ref{prop:dual-short});
on the Shapes pair the unshifted point has $\phi_\se=2.1\cdot10^6$ against the optimal $0.101$, so
the shift is not cosmetic.

\begin{table}[htbp]
\centering\small
\caption{The two-marginal block partition on DOTmark, $4000$ passes. Both block columns are ours:
``analysed'' is Algorithm~\ref{alg:main-short} with its own constants; ``rescaled'' divides the gap term by the
measured $\bar\Delta/F_\se(\lambda_0)$ and is not certified. A parenthesized entry is the best
relative dual gap reached instead of the target. Half-steps are the comparable axis: a block pass and
a Sinkhorn half-step each read the cost matrix once. ``Slack'' is the smallest ratio
$\bar L_{k,B}/\|\nabla^2_{BB}\phi_\se(y_k)\|_2$ over $300$ steps; the potential of
Theorem~\ref{thm:phase-short} decreased in conditional expectation at every one of them, which at this
slack is not a test of the constants. One run per pair, at seed $0$. Sinkhorn sweeps are logged on
this run's finer grid, so they differ by one logging step from the counts quoted in
Section~\ref{sec:exp-short} ($35$ against $40$ here, $45$ against $50$ on GRFmoderate).}
\label{tab:block-partition}
\begin{tabular}{@{}lrrrrrr@{}}
\toprule
& \multicolumn{2}{c}{block passes to $10^{-6}$} & Sinkhorn & \multicolumn{1}{c}{half-steps} & & \\
\cmidrule(lr){2-3}
Pair & analysed & rescaled & sweeps & vs.\ Sinkhorn & $S_0$ ratio & slack \\
\midrule
ClassicImages    & $(3.5\cdot10^{-6})$ & $2750$              & $60$ & $22.9$ & $1.24$ & $8.0\cdot10^{5}$ \\
Shapes           & $45$                & $45$                & $10$ & $2.2$  & $1.60$ & $5.3\cdot10^{5}$ \\
CauchyDensity    & $750$               & $(3.0\cdot10^{-6})$ & $50$ & $7.5$  & $3.77$ & $4.5\cdot10^{5}$ \\
GRFmoderate      & $(1.9\cdot10^{-6})$ & $1695$              & $45$ & $18.8$ & $1.34$ & $6.8\cdot10^{5}$ \\
MicroscopyImages & $130$               & $160$               & $35$ & $1.9$  & $3.49$ & $6.4\cdot10^{5}$ \\
\bottomrule
\end{tabular}
\end{table}

\begin{figure}[t]
\centering
\includegraphics[width=\linewidth]{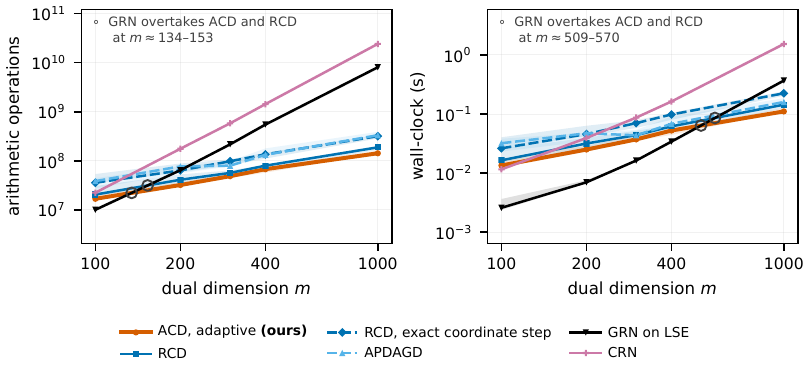}
\caption{Cost to relative dual gap $10^{-9}$ against dual dimension, synthetic family ($n=4m$,
$\bar d=75$, $\gamma=0.05$); medians over five seeds with min--max bands, both axes logarithmic.
Second-order curves rise with slope $\approx3$ and coordinate curves with slope $\approx1$, so each
second-order/coordinate pair crosses once. Open circles mark where GRN crosses ACD and RCD, located
by log-log interpolation between the sampled dimensions. Different pairs cross at different $m$:
CRN is already the more expensive of the two branches in arithmetic at $m=100$, so its arithmetic
crossing lies below the range plotted.}
\label{fig:crossover}
\end{figure}

\section{Experimental protocol}\label{app:exp-full}
\paragraph{Instances.}
Synthetic sparse entropy-regularized LPs are drawn with column nonzero counts
$\max(1,\mathrm{Bin}(m,s_{\rm col}/m))$, where the target mean row degree $\bar d$ fixes the column sparsity $s_{\rm col}=\mathrm{round}(\bar d\,m/n)$, entries $\mathrm{Unif}(0.5,1.5)$,
and columns rescaled so that $A^\top q=\mathbf1$ for $q=\mathbf1/m$; a planted feasible
$x^{\rm f}\sim\mathrm{Dir}(\mathbf1)$ gives $b=Ax^{\rm f}$, and $c\sim\mathrm{Unif}(0,1)$. We use
$n=4m$, $\bar d=75$, $\gamma=0.05$, seeds $0$--$4$; the main instance is $m=300$ and the density
ablation fixes $m=1000$ with $\bar d\in\{10,25,75,150,300\}$. DOTmark
\citep{schrieber2017dotmark} uses $32\times32$ images ($m=2048$, $n=2^{20}$) with squared Euclidean
cost normalized to maximum $1$ and $\gamma=0.01$, on pairs $(1,2)$ of ClassicImages, Shapes,
CauchyDensity, GRFmoderate and MicroscopyImages.

\paragraph{Common start and metric.}
LSE methods start at $\lambda_0=0$; SE methods start at $\lambda_0=t^\star(0)q=-\gamma\log Z(0)q$,
which by Proposition~\ref{prop:dual-short} gives $x(\lambda_0)=p(0)$ and hence a common initial
dual gap for both families. Every curve reports the best-so-far relative dual gap
$(\phi(\lambda)-\phi^\star)/|\phi_\se^\star|$, with each method measured against its own dual.
Because $F_\se\ge F_\lse$ by \eqref{eq:gap-split}, this holds the SE methods to the stricter of the
two criteria. Medians are over five seeds (or five image pairs). Min--max bands are drawn for the adaptive ACD and
RCD curves on the epoch and arithmetic axes; the wall-clock axis and the DOTmark figure carry none. The reference optimum comes from L-BFGS-B on the LSE dual followed by a GRN polish, reaching
an SE stationarity residual with median $2\cdot10^{-14}$ and worst case $1.7\cdot10^{-12}$ across
all runs; the $10^{-12}$ level is therefore at the
resolution floor and should not be read as a converged digit.

\paragraph{Methods and inputs.}
ACD uses $\delta=1/3$ and $\bar\Delta=\phi_\se(\lambda_0)+f_s(x^{\rm f})$ from the planted point.
The radius is taken as $r=2\|\lambda_0-\lambda^\star_\se\|_2$ \emph{from the reference solution}:
the runs therefore use an oracle radius, not the computable bound of
Proposition~\ref{prop:input-cert}, and Section~\ref{sec:limits} should be read with that in mind.
The adaptive variant rescales the baseline $\kappa$ by backtracking on validity at the sampled
coordinates: per epoch, divide by $4$ if no violation occurred and the minimum observed ratio was
at least $4$, and multiply by $4$ (capped at the analysed value) after any violation. This rule has no
guarantee; its violation counts are reported in Table~\ref{tab:cert-sweep}. RCD is Algorithm~4 of
\citet{lobanov2024} with uniform sampling and row-wise constants. GRN, CRN, APDAGD and Sinkhorn use
their theoretical constants with adaptive backtracking, with Sinkhorn run in the log domain.

\paragraph{Arithmetic model.}
A coordinate update is charged per nonzero of the sampled row: $5$ operations for ACD (three to
form the gradient and the local curvature, two to update the lazy state) and $4$ for RCD, plus
$3$ per nonzero per inner Newton iteration when the exact block step is used, capped at $24$
iterations. These per-nonzero constants are the same order but are not calibrated against each
other: counting elementary operations in the two loop bodies gives roughly $15$ for an ACD
coordinate against $6$ for an RCD one, so the model favours ACD by about a factor two on this axis.
It is adequate for growth exponents, which is what we fit from it, and not for margins near unity;
where such a margin matters we quote wall-clock. A matrix--vector product with $A$ costs $1$ per nonzero; a Cholesky factorization
costs $m^3/3$, a symmetric eigendecomposition $m^3$, and the sparse Hessian assembly
$\sum_j\|a_j\|_0^2+m^2$; objective evaluations at logging points are charged as incurred. Wall-clock is indicative only: all runs were executed on one laptop, single process and
single-threaded BLAS, so the timings describe that one machine, and fitted wall-clock exponents
should be read to one decimal.

\begin{figure}[t]
\centering
\includegraphics[width=\linewidth]{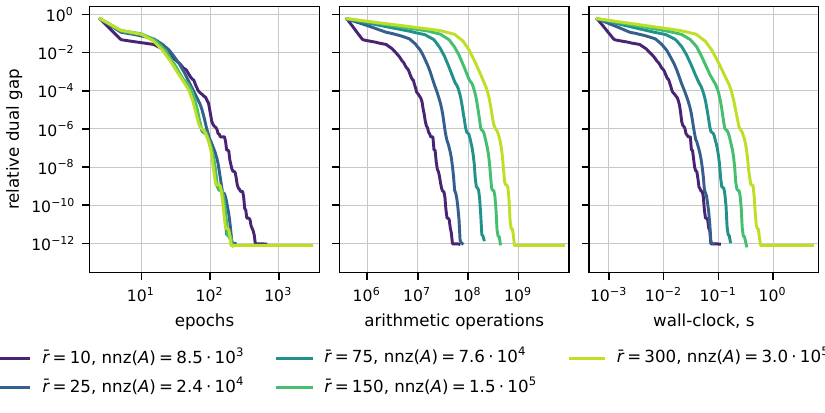}
\caption{Density ablation at fixed dual dimension. Coordinate epochs are almost density
independent, while arithmetic and wall-clock costs track the number of touched nonzeros.}
\label{fig:density-short}
\end{figure}

\subsection{The analysed schedule against its adaptive rescaling}
Table~\ref{tab:cert-sweep} reports the sweep described in Section~\ref{sec:exp-short}, on the
$m=1000$ synthetic instance at target $10^{-6}$ with a budget of $3000$ epochs, against $345$
epochs for RCD and $175$ for RCD with the exact coordinate step. The final column is the smallest observed ratio between the bound
$\kappa_B+\nu_BF$ and the curvature actually encountered at the sampled block, so it measures how
much slack the analysis carries on these instances.

\begin{table}[htbp]
\centering\small
\caption{The analysed schedule (ours) against rescalings of its baseline, $m=1000$ synthetic instance;
medians over five seeds, target $10^{-6}$ relative, budget $3000$ epochs; RCD needs $345$ and
RCD with an exact coordinate step $175$.
``scale'' multiplies $\kappa_B$ inside $\bar L_{k,B}$, so $1$ is Theorem~\ref{thm:hg-short} as
proved and a scale of $10^{-4}$ shrinks $\sqrt{\kappa_B}$, and with it the accuracy channel of $M_k$,
by $10^{-2}$; the growth channel $C_1$ is left untouched throughout;
``epochs'' is when the \emph{iterate} first reaches the target, by the same rule applied to RCD;
``$10^{-12}$'' counts seeds reaching the stopping tolerance within budget; the schedule's own
termination test $T_k\ge\Pi/\eps$ is met at none of the thirty settings, so every epoch count in this
sweep is an iterate crossing and never a completed phase;
``slack'' is the smallest observed ratio of envelope to realized curvature. The full
$5\times6$ sweep over radius and scale is Table~\ref{tab:cert-sweep-full}.}
\label{tab:cert-sweep}
\begin{tabular}{@{}llrrrr@{}}
\toprule
$r/\|\lambda_0-\lambda^\star_\se\|_2$ & scale & epochs & $10^{-12}$ & violations & slack \\
\midrule
$0.5$ & $1$      & $1720$ & $5/5$ & $0$ & $13.6$\\
$1$ & $1$        & ---    & $0/5$ & $0$ & $77.2$\\
$1$ & $10^{-2}$  & $280$  & $5/5$ & $0$ & $2.9$\\
$1$ & $10^{-4}$  & $125$  & $5/5$ & $0$ & $1.9$\\
\midrule
$2$  & $1$ & ---    & $0/5$ & $0$ & $1.5\cdot10^{3}$\\
$5$  & $1$ & ---    & $0/5$ & $0$ & $1.1\cdot10^{4}$\\
$20$ & $1$ & $205$  & $0/5$ & $0$ & $2.9\cdot10^{4}$\\
\bottomrule
\end{tabular}
\end{table}

Three readings. The inequality holds: at the analysed scale no violation was recorded anywhere in
the sweep, on any seed. It is loose, and how loose depends entirely on the radius: the slack runs
from $13.6$ at $r=0.5$ to $2.9\cdot10^4$ at $r=20$, because a larger $r$ inflates $\Pi$ and hence
the gap level at which the envelope is evaluated. And the looseness is what the schedule pays:
at the analysed scale the $10^{-12}$ tolerance is reached only at $r=0.5$, no setting completes a
phase, and competitiveness with RCD needs $\kappa$ shrunk by $10^{2}$ to $10^{4}$, at which point the method is no longer the one analysed in
Section~\ref{sec:alg-short}.

\begin{figure}[htbp]
\centering
\includegraphics[width=\linewidth]{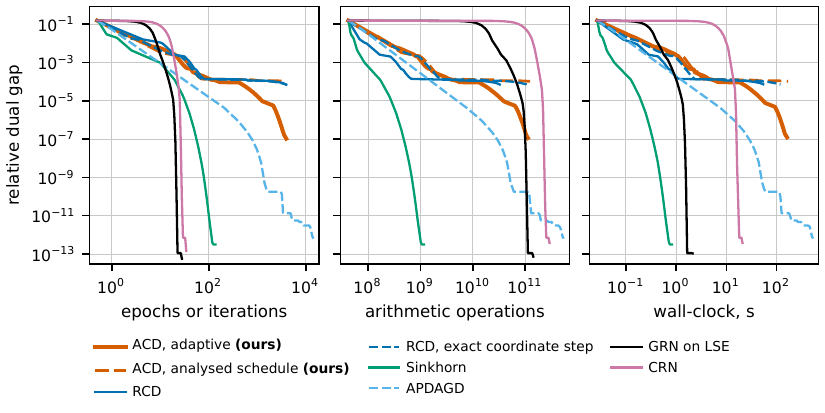}
\caption{DOTmark $32\times32$ transport ($m=2048$, $n=2^{20}$, $\gamma=0.01$), five image pairs.
Sinkhorn dominates on every axis, by minimizing exactly over each $1024$-dimensional marginal block
in closed form: where both methods finish, by $32$ to $71$ on epochs and $230$ to $494$ on
wall-clock.}
\label{fig:dotmark-short}
\end{figure}

\paragraph{Provenance of the reported numbers.}
Every figure and table in the paper is read off one set of runs under the
protocol above, Figures~\ref{fig:synth-short}--\ref{fig:dotmark-short} included. Two quantities are
not. The fitted growth exponents of Table~\ref{tab:crossover} and the two-block query estimate
quoted in Section~\ref{sec:exp-short} are post-processing of the recorded per-iteration traces
rather than outputs of a run. And the deterministic checks of the algebraic identities are not
optimization runs at all; Appendix~\ref{app:fixed-support-protocol} states what they verify.

\subsection{The full radius-by-scale sweep}
Table~\ref{tab:cert-sweep-full} is the sweep abbreviated in Table~\ref{tab:cert-sweep}. Reading
down a column shows the effect of rescaling the analysed baseline at a fixed radius; reading across
a row shows that the slack grows with the supplied radius. The epoch count does not. The schedule's own
termination test is met at none of the thirty settings, so every entry records the first time the
\emph{iterate} passes the target under a schedule that would not have stopped there; the $205$ at
$r=20$, scale $1$, is that reading and not a faster method. No violation of the envelope was
recorded at any of the thirty settings, on any seed.

\begin{table}[htbp]
\centering\small
\caption{Epochs to relative dual gap $10^{-6}$ for the analysed schedule (ours), $m=1000$, medians
over five seeds, budget $3000$ epochs; ``---'' means at least one seed did not reach the target, and
RCD needs $345$. The row
under ``scale $=1$'' gives the smallest observed ratio of envelope to realized curvature.}
\label{tab:cert-sweep-full}
\begin{tabular}{@{}lrrrrr@{}}
\toprule
scale & $r/\|\cdot\|=0.5$ & $1$ & $2$ & $5$ & $20$ \\
\midrule
$1$        & $1720$ & ---    & ---    & ---    & $205$ \\
\emph{slack} & \emph{$13.6$} & \emph{$77.2$} & \emph{$1.5\!\cdot\!10^3$} & \emph{$1.1\!\cdot\!10^4$} & \emph{$2.9\!\cdot\!10^4$} \\
$10^{-1}$  & $530$  & $880$  & $1430$ & $2545$ & ---   \\
$10^{-2}$  & $230$  & $280$  & $425$  & $710$  & $355$ \\
$10^{-3}$  & $150$  & $155$  & $190$  & $195$  & $310$ \\
$10^{-4}$  & $120$  & $125$  & $135$  & $150$  & $210$ \\
$10^{-5}$  & $110$  & $115$  & $115$  & $140$  & $205$ \\
\bottomrule
\end{tabular}
\end{table}

\subsection{A known-solution generator}
This is a proposed protocol, not a claim about completed runs. Construct a sparse $A$ with
$A^\top q=\mathbf1$: one concrete nonnegative construction splits the rows into two nonempty
families and makes each column sum to one within each family, using random nonempty supports and
normalized positive weights, so that $q=(\mathbf1,0)$ works. For a signed variant, add signed
zero-sum perturbations within a family's support, using at least two selected rows so that mixed
signs do not change the column sums. The support law, weight distribution, perturbation law and
seeds must be fixed and reported. Then choose $x^\star>0$ with $\mathbf1^\top x^\star=1$, a finite
$\lambda^\star_\se$, and set
\begin{equation}\label{eq:known-solution-generator}
b:=Ax^\star,\qquad c:=A^\top\lambda^\star_\se-\gamma\log x^\star .
\end{equation}
Then $q^\top b=1$, strict feasibility holds, and $x(\lambda^\star_\se)=x^\star$, so by stationarity
and convexity this is an exact SE minimizer, and its projection onto $\mathcal H$ is an LSE
minimizer. Any chosen start $z$ then has the exactly valid radius
$r=\|z-\lambda^\star_\se\|_2$, and a zero LSE start has $R=\|P_{\mathcal H}\lambda^\star_\se\|_2$.
This supplies a reference solution without running a numerical optimizer.

\subsection{Metrics and numerical accuracy}
Report the full residual $\|Ap(\lambda)-b\|_2$ and the common normalized gap $F_\lse(\lambda)$ for
every method. With a known optimum,
\[
F_\lse(\lambda)=\gamma\KL(x^\star\|p(\lambda)),\qquad
F_\se(\lambda)=\gamma\sum_jx_j^\star h(t_j),\quad t_j=a_j^\top(\lambda-\lambda^\star_\se)/\gamma .
\]
Use log-sum-exp for normalization, compensated or higher-precision summation where cancellation
matters, and a series or \texttt{expm1}-based evaluation of $h(t)$ near zero. To support a reported
$10^{-12}$ gap, verify that the reference error is well below that threshold rather than estimating
it by subtracting nearly equal objective values. The figures divide the gap by $|\phi_\se^\star|$, a
per-instance constant shared by all methods; the ratio $F_\lse(\lambda)/F_\lse(\lambda_0)$ at a
common start is a normalization-invariant alternative.

\paragraph{Where the implementation still differs from the analysed method.}
Every run but the two-marginal transport experiment of
Appendix~\ref{app:cert-ot} is the $|B|=1$ specialization of
Theorem~\ref{thm:work-short}, and no run chooses a partition: the one we test is the one transport
hands us. Each analysed run keeps the anchor fixed and drives a single phase to the final target,
which is the fixed-anchor repetition of Theorem~\ref{thm:work-short}, and RCD is reported both with
the envelope step of its own analysis and with the same exact coordinate solve our variants take, so
that epoch counts are comparable. Two deviations remain, both in favour of the analysed method: the runs report a single phase,
which is the high-probability guarantee of Theorem~\ref{thm:phase-short} rather than the
expectation guarantee of Theorem~\ref{thm:work-short}, and the phase loop is bounded by the epoch
budget rather than by $J_{\rm rep}$.

\paragraph{Validity tests and violations of the adaptive rule.}
The adaptive variant checks the envelope at each sampled coordinate and rescales $\kappa$ from
that. On the $m=300$ instance it performed $312\,107$ such tests over five seeds and recorded
\emph{no} violation; on $m=1000$, $1\,010\,637$ tests and again none. The scale it settles on is
$\kappa=1.53\cdot10^{-5}$, reached by dividing by $4$ once per epoch for eight epochs, and the
smallest observed ratio of envelope to realized curvature falls with it, from $2.0\cdot10^4$ at
$\kappa=1$ to $73$ after five epochs. This is evidence that the rule is not reckless on these
instances; it is not a guarantee, because the test only ever sees the sampled coordinate.

\subsection{Fixed-constant versus adaptive runs}
A direct validation must compare the unchanged fixed-constant wrapper of
Theorem~\ref{thm:work-short}, the adaptive heuristic and RCD on identical instances and starts,
reporting the radius and gap inputs, sampling laws, stopping rules, restart behavior, and all
preparation, phase-boundary and output work. Changing a sampled block's constant and then reusing
that same sample under new sampling weights needs its own conditional-expectation analysis and is
not covered by a proof that fixes every constant beforehand. For each adaptive run, log the number
of validity tests,
violations, rejections and restarts, the violation magnitudes and the recovery rule. A convenient
observable is the sampled descent residual
$f(y-U_B\nabla_Bf(y)/\bar L_B)-f(y)+\|\nabla_Bf(y)\|_2^2/(2\bar L_B)$, together with the
floating-point tolerance used to call it positive; passing this local test does not certify all
blocks or the global accelerated guarantee. Further useful ablations compare signed against
nonnegative rows, uniform against curvature-weighted sampling, and different block partitions, all
measured in total arithmetic rather than query counts.

\section{Smoothness models and their precise implications}\label{app:smoothness-proofs}
All statements here assume a convex $C^2$ function on $\R^m$ with an attained minimum and
$F=f-f^\star$; the Hessian conditions are pointwise global statements. A first-order condition
posed only on a neighbourhood also requires its radius to be tracked, and the two formulations are
not identified without that step.

\subsection{Three distinct conditions}
For a convex $C^2$ function, the following are genuinely different:
\begin{align}
\partial_{ii}^2f(u)&\le L_{0,i}+L_{1,i}|\partial_if(u)|,
&&\text{coordinate-gradient control},\label{eq:lc-model}\\
\partial_{ii}^2f(u)&\le L_{0,i}+L_{1,i}\|\nabla f(u)\|_2,
&&\text{full-gradient control of a diagonal},\label{eq:lf-model}\\
\partial_{ii}^2f(u)&\le\kappa_i+\nu_iF(u),
&&\text{gap-dependent control (BGG)}.\label{eq:hc-model}
\end{align}
Condition \eqref{eq:lc-model} implies \eqref{eq:lf-model} with the same constants, but the converse
fails even for an SE dual, as the example below shows. The \emph{global} $(L_0,L_1)$ model bounds
$\|\nabla^2f\|_2$ rather than only its diagonal. Our accelerated theorem uses \eqref{eq:hc-model}
and does not silently assume \eqref{eq:lc-model} for signed matrices.

\subsection{From coordinate-gradient control to a gap bound}
\begin{lemma}[A coordinate self-bound]\label{lem:lc-to-hc}
Under \eqref{eq:lc-model}, for every $u$ and every coordinate $i$,
\begin{align}
|\partial_if(u)|^2&\le2F(u)\bigl(L_{0,i}+L_{1,i}|\partial_if(u)|\bigr),\label{eq:coordinate-selfbound}\\
\partial_{ii}^2f(u)&\le2L_{0,i}+2L_{1,i}^2F(u).\label{eq:lc-to-hc-exact}
\end{align}
\end{lemma}
\begin{proof}
Write $g=\partial_if(u)$ and $M=L_{0,i}+L_{1,i}|g|$. For $g\ne0$, $M>0$, move along
$u-t\operatorname{sign}(g)e_i$. Until the directional derivative first reaches zero its magnitude
is at most $|g|$ by convexity, so the second derivative is at most $M$; a zero cannot occur before
$t=|g|/M$, since changing the initial derivative $-|g|$ to zero at rate at most $M$ takes at least
that distance. The quadratic upper model is therefore valid up to $|g|/M$, giving
$f(u-ge_i/M)\le f(u)-g^2/(2M)$, and comparison with $f^\star$ proves
\eqref{eq:coordinate-selfbound}. For $g=0$ it is immediate; if $g\ne0$ and $M=0$ then
$L_{0,i}=L_{1,i}=0$ and $f$ is affine with nonzero slope on that line, contradicting attainment.

If $L_{1,i}=0$, \eqref{eq:lc-to-hc-exact} is direct. Otherwise
$(M-L_{0,i})^2=L_{1,i}^2g^2\le2L_{1,i}^2F(u)M$, and for $M>0$ expanding and dividing by $M$ gives
$M\le2(L_{0,i}+L_{1,i}^2F(u))-L_{0,i}^2/M\le2(L_{0,i}+L_{1,i}^2F(u))$. Since
$\partial_{ii}^2f(u)\le M$ the claim follows; $M=0$ is immediate from Hessian nonnegativity.
\end{proof}
The proof uses one-dimensional monotonicity of a \emph{partial} derivative; replacing
$|\partial_if|$ by the full gradient norm inside it is not justified, since the other components
can change along the coordinate line. Conversely, positive semidefiniteness and the trace bound
show that \eqref{eq:lc-model} implies a global norm model with $L_0=\sum_iL_{0,i}$ and
$L_1=(\sum_iL_{1,i}^2)^{1/2}$, that is, aggregated constants rather than the original
per-coordinate ones.
Applying Theorem~\ref{thm:phase-short} through \eqref{eq:lc-to-hc-exact} yields, at fixed failure
level, $\widetilde O(1+r\sum_i\sqrt{L_{0,i}}/\sqrt\eps+r\sum_iL_{1,i})$ queries. That is a
specialization, not a better $L_1$ exponent.

\subsection{A local upper model from the BGG condition}
\begin{lemma}[Explicit local radius]\label{lem:hessian-local-radius}
Suppose $0\le\partial_{ii}^2f(u)\le\kappa_i+\nu_iF(u)$ everywhere. If $\nu_i>0$ and
$|t|\le1/(4\sqrt{\nu_i})$, then
\begin{equation}\label{eq:hessian-local-model}
f(u+te_i)\le f(u)+t\,\partial_if(u)+(\kappa_i+\nu_iF(u))t^2 .
\end{equation}
For $\nu_i=0$ the same holds for every $t$ with the sharper coefficient $\kappa_i/2$.
\end{lemma}
\begin{proof}
The block descent self-bound of Appendix~\ref{app:acc-proof} gives
$|\partial_if(u)|^2\le2F(u)(\kappa_i+\nu_iF(u))$. For $\nu_i>0$ let
$\psi(t)=F(u+te_i)+\kappa_i/\nu_i$; then $\psi\ge0$, $\psi''\le\nu_i\psi$ and
$|\psi'(0)|\le\sqrt{2\nu_i}\,\psi(0)$. Integral comparison with the solution of $y''=\nu_iy$ on
either side of zero gives
$\psi(t)\le\psi(0)(\cosh(\sqrt{\nu_i}|t|)+\sqrt2\sinh(\sqrt{\nu_i}|t|))<2\psi(0)$ for
$0<|t|\le1/(4\sqrt{\nu_i})$ (non-strict if $\psi(0)=0$). Integrating
$\partial_{ii}^2f(u+te_i)\le2(\kappa_i+\nu_iF(u))$ proves the model, and since the constant is
$\cosh\tfrac14+\sqrt2\sinh\tfrac14=1.389$ the factor $2$ could be replaced by $1.39$; $\nu_i=0$ is the global
one-dimensional smoothness inequality.
\end{proof}
This makes the comparison with coordinate gap-dependent assumptions explicit: only absolute
curvature and radius factors change, and a coordinate condition is still not identified with a
full-gradient one.

\subsection{SE constants and signed rows}
For $A\ge0$, using $A_{ij}^2\le\rho_iA_{ij}$,
\begin{equation}\label{eq:se-l01-coordinate}
\partial_{ii}^2\phi_\se=\gamma^{-1}\sum_jA_{ij}^2x_j\le\frac{\rho_i}{\gamma}(b_i+\partial_i\phi_\se)
\le\frac{\rho_ib_i}{\gamma}+\frac{\rho_i}{\gamma}|\partial_i\phi_\se|,
\end{equation}
so $L_{0,i}=\rho_ib_i/\gamma$ and $L_{1,i}=\rho_i/\gamma$: the baseline combines the target
constraint value with the largest row coefficient, and $L_{1,i}^{-1}$ is the coordinate
displacement over which factor weights change substantially. The same holds for a nonpositive row
after flipping signs, but not for a row with both signs. For arbitrary signs the mass identity
$q^\top\nabla\phi_\se=Z-1$ gives instead the valid global bound
\begin{equation}\label{eq:se-l01-global}
\|\nabla^2\phi_\se\|_2\le\gamma^{-1}\sum_jx_j\|a_j\|_2^2\le\frac{\rho^2}{\gamma}Z
\le\frac{\rho^2}{\gamma}+\frac{\rho^2\|q\|_2}{\gamma}\|\nabla\phi_\se\|_2 .
\end{equation}
These sign-free constants can be loose and do not imply \eqref{eq:lc-model}.

To see that the gap is real, take $\gamma=1$, $c=0$,
$A=\begin{psmallmatrix}1&1\\1&-1\end{psmallmatrix}$, $b=(1,0)^\top$, $q=(1,0)^\top$. The strictly
feasible primal point is $(1/2,1/2)$ and the SE dual is $f(u,v)=e^{u+v}+e^{u-v}-u$ with minimizer
$(-\log2,0)$. At $v=0$ we have $\partial_vf=0$ while $\partial^2_{vv}f=2e^u$, so no finite constants
satisfy \eqref{eq:lc-model} in that coordinate, yet \eqref{eq:se-l01-global} still holds. The
sign-free theorem of Section~\ref{sec:hg-short} therefore covers objectives that the strong
coordinate model excludes.

\section{Computable geometry and the balanced example}\label{app:geometry-proofs}
\subsection{Effective dimension and its crossover}\label{sec:geometry}
Summing the singleton envelopes of Theorem~\ref{thm:hg-short} at $\theta=2$ turns $(\kappa,\nu)$ into
one gap-dependent dimension $d_{\rm curv}(\Delta)$, interpolating from the effective rank
$\tr H^\star/\|H^\star\|_2$ at $\Delta=0$ to the row spread $d_A=\sum_i\rho_i^2/\rho^2$ as
$\Delta\to\infty$, with $\rho_i=\|A_{i:}\|_\infty$ and $\rho=\max_j\|a_j\|_2$. It is what decides
square-root sampling for a fixed Hessian, and measuring both ends calibrates the motivation for coordinate
methods: on our family the effective rank grows in proportion to $m$, so the
predicted gain is a constant $4.2$ rather than $\sqrt m$, and on transport it is $1.6$ to $3.4$ --
the instances with the least spectral asymmetry to exploit, which is part of why
Section~\ref{sec:exp-short} loses to Sinkhorn there. Appendix~\ref{app:geometry-proofs} carries the
proof, the crossover, the balanced family and the full measurement.

\subsection{The gap-dependent dimension, and what it measures}
The interpolation is exact: $d_{\rm curv}(\Delta)$ equals $\tr H^\star/\|H^\star\|_2$ at $\Delta=0$,
tends to $d_A$ as $\Delta\to\infty$, exchanges the two at
$\Delta_{\rm cross}=\gamma^2\log2\,\|H^\star\|_2/\rho^2$, and has both endpoints in $[1,m]$. The
matching statement about the method is that $\nu_B\equiv0$ gives $C_1=0$ and
$N=\widetilde O(1+rC_0/\sqrt{\delta\eps})$, exactly the accelerated coordinate bound of
\citet{allenzhu2016,nesterovstich2017}: generalized smoothness recovers the classical rate when its
growth channel vanishes.

Measured directly on the generated instances, on the
synthetic family $d_H$ grows in proportion to $m$, from $5.7$ at $m=100$ to $57.9$ at $m=1000$, so
$\sqrt{m/d_H}$ is a constant $4.2$ across the sweep; it moves with density instead, from $2.4$ at
$\bar d=10$ to $8.3$ at $\bar d=300$. On DOTmark $d_A=m/2$ exactly, since an incidence matrix has
$\rho_i=1$ and $\rho=\sqrt2$, while $d_H$ ranges over $175$ to $821$, giving $1.6$ to $3.4$.
$\Delta_{\rm cross}$ is $2\cdot10^{-4}$ to $3\cdot10^{-3}$ on the synthetic family, far below the gaps
the runs traverse, so the method spends its time at the $d_A$ end.

\subsection{The gap-dependent dimension}
With $\rho_i=\|A_{i:}\|_\infty$, $\rho=\max_j\|a_j\|_2$ and $\theta=2$ in
Theorem~\ref{thm:hg-short}, the singleton envelopes on $\{F\le\Delta\}$ are
\begin{equation}\label{eq:canonical-envelopes}
\ell_i(\Delta)=2\log2\,H^\star_{ii}+2\Delta\rho_i^2/\gamma^2,\qquad
\ell_{\rm full}(\Delta)=2\log2\,\|H^\star\|_2+2\Delta\rho^2/\gamma^2 .
\end{equation}

\begin{proposition}[Effective dimension and its crossover]\label{prop:crossover}
Assume $H^\star\ne0$ and define
\begin{equation}\label{eq:dcurv-short}
d_{\rm curv}(\Delta):=\frac{\sum_i\ell_i(\Delta)}{\ell_{\rm full}(\Delta)},\qquad
d_H:=\frac{\tr H^\star}{\|H^\star\|_2},\qquad
d_A:=\frac{\sum_i\rho_i^2}{\rho^2}.
\end{equation}
With $\Delta_{\rm cross}:=\gamma^2\log2\,\|H^\star\|_2/\rho^2$,
\begin{equation}\label{eq:dcurv-interp-short}
d_{\rm curv}(\Delta)=
\frac{\Delta_{\rm cross}}{\Delta_{\rm cross}+\Delta}\,d_H+
\frac{\Delta}{\Delta_{\rm cross}+\Delta}\,d_A .
\end{equation}
Both $d_H$ and $d_A$ lie in $[1,m]$, so $d_{\rm curv}(0)=d_H$ and
$d_{\rm curv}(\Delta)\to d_A$ as $\Delta\to\infty$.
\end{proposition}

\noindent The identity follows by summing \eqref{eq:canonical-envelopes}. Above
$\Delta_{\rm cross}$ the relevant dimension is $d_A$, which asks how the largest row coefficients
are spread across factors; below it $d_H$ takes over. A balanced incidence family on which
$d_H=d_A=m/s$ collapses the interpolation is given in Proposition~\ref{prop:balanced-family}.

\subsection{Remarks on the two dimensions}
First, the range claimed in Proposition~\ref{prop:crossover}. For a nonzero PSD matrix,
$\|H^\star\|_2\le\tr H^\star\le m\|H^\star\|_2$, giving $d_H\in[1,m]$; and since
$\rho^2=\max_j\|a_j\|_2^2\ge\max_{i,j}A_{ij}^2$ while
$\sum_i\rho_i^2=\sum_i\max_jA_{ij}^2\ge\max_j\|a_j\|_2^2$, we get
$\rho^2\le\sum_i\rho_i^2\le m\rho^2$ and hence $d_A\in[1,m]$. A convex combination of the two then
lies in $[1,m]$ as well. Three further points deferred from Section~\ref{sec:geometry}. First, $d_A$ is not invariant under a change
of dual coordinates, whereas $d_H$ is invariant under orthogonal ones (a non-orthogonal rescaling
changes $\tr H^\star/\|H^\star\|_2$ as well): the quantity $\sum_i\rho_i^2$ depends on the basis in which
the rows are expressed, so both the coordinate system and the data-access costs matter when reading
$d_{\rm curv}$. Second, the two terms of $\ell_i$ scale differently in the regularization, as
$\gamma^{-1}$ for the baseline and $\gamma^{-2}$ for the growth coefficient; since $H^\star$ also
depends on $\gamma$, individual entries need not be monotone in $\gamma$ and the crossover level
$\Delta_{\rm cross}$ moves with it. Third, \eqref{eq:dcurv-interp-short} is stated for singleton
coordinates on purpose: block operator norms do not sum to a trace, so while
Theorem~\ref{thm:hg-short} and the algorithm remain valid for any partition, the literal
effective-rank identity does not survive blocking. The cost-weighted ratio below is the quantity
that does generalize.

\begin{corollary}[Exact constants for scaled indicator rows]\label{cor:indicator-exact}
If $A_{ij}\in\{0,\rho_i\}$ with $\rho_i>0$, then $A_{ij}^2=\rho_iA_{ij}$, so
$H^\star_{ii}=\rho_i\sum_jA_{ij}x^\star_j/\gamma=\rho_ib_i/\gamma$ and
$\kappa_i=(2\log2/\gamma)\rho_ib_i$ uses the exact diagonal of $H^\star$ with no knowledge of
$x^\star$. Incidence matrices, and hence optimal transport, are of this form.
\end{corollary}

\subsection{Cost-weighted work ratio}
For a convex quadratic with Hessian $H$ the coordinate curvatures are $L_i=H_{ii}$ and the
full-gradient curvature is $L=\|H\|_2$, so square-root importance sampling \citep{allenzhu2016}
compares $\sum_i\sqrt{H_{ii}}$ against $m\sqrt{\|H\|_2}$:
\begin{equation}\label{eq:classical-rank-ratio}
\frac{\sum_i\sqrt{H_{ii}}}{m\sqrt{\|H\|_2}}\le\sqrt{\frac{d_H}{m}},\qquad
d_H:=\frac{\tr H}{\|H\|_2},
\end{equation}
which is the classical effective-rank criterion that Section~\ref{sec:geometry} generalizes.
With $d_i=1+\nnz(A_{i:})$ and $d_{\rm full}=\sum_id_i$ as the nominal full-pass cost, define
\begin{equation}\label{eq:work-ratio-geometry}
\mathcal R_{\rm work}(\Delta):=\frac{\sum_id_i\sqrt{\ell_i(\Delta)}}{d_{\rm full}\sqrt{\ell_{\rm full}(\Delta)}} .
\end{equation}
For equal $d_i$, \eqref{eq:classical-rank-ratio} gives
$\mathcal R_{\rm work}(\Delta)\le\sqrt{d_{\rm curv}(\Delta)/m}$; for heterogeneous costs it also
measures whether the stiff coordinates are the expensive ones. Replacing $H^\star_{ii}$ and
$\|H^\star\|_2$ by chosen upper bounds gives the \emph{implemented} comparison, and a ratio of
loose bounds is not an intrinsic dimension.

\begin{proposition}[Accuracy and growth contributions]\label{prop:work-coefficients}
Choose valid singleton constants $(\kappa_i,\nu_i)$ and full-block constants
$(\kappa_{\rm f},\nu_{\rm f})$ with $\kappa_{\rm f},\nu_{\rm f}>0$, and set
$G_0:=S_0/(d_{\rm full}\sqrt{\kappa_{\rm f}})$ and $G_1:=S_1/(d_{\rm full}\sqrt{\nu_{\rm f}})$. For
$r,\eps>0$ and the same start, radius, gap bound, target and confidence, the ratio of the two
leading work expressions of Theorem~\ref{thm:work-short} and its one-block specialization is
\begin{equation}\label{eq:two-channel-work}
\frac{rS_0/\sqrt\eps+rS_1}{rd_{\rm full}(\sqrt{\kappa_{\rm f}}/\sqrt\eps+\sqrt{\nu_{\rm f}})}
=\omega_\eps G_0+(1-\omega_\eps)G_1,\qquad
\omega_\eps:=\frac{\sqrt{\kappa_{\rm f}}/\sqrt\eps}{\sqrt{\kappa_{\rm f}}/\sqrt\eps+\sqrt{\nu_{\rm f}}} .
\end{equation}
If all growth coefficients vanish, only $G_0$ remains.
\end{proposition}
The identity is algebraic; that these expressions occur in valid bounds follows by applying
Theorem~\ref{thm:work-short} to the two partitions. Logarithms, initialization, output and additive
sampling terms stay in the full bounds and their ratio is not claimed to equal
\eqref{eq:two-channel-work}. This is neither a runtime lower bound nor a claim of optimality among
sampling laws.

\subsection{The balanced incidence family}
\begin{proposition}[Balanced incidence family]\label{prop:balanced-family}
Let $A\in\{0,1\}^{m\times n}$ have exactly $s_{\rm col}$ ones in every column, let $b=(s_{\rm col}/m)\mathbf1$, and
assume strict feasibility. Then $q=\mathbf1/s_{\rm col}$ certifies the redundant mass equation,
$H^\star_{ii}=s_{\rm col}/(m\gamma)$, $\|H^\star\|_2=s_{\rm col}^2/(m\gamma)$, and
\begin{equation}\label{eq:balanced-geometry}
d_H=d_A=d_{\rm curv}(\Delta)=m/s_{\rm col}\qquad(\Delta\ge0),
\end{equation}
with the bounds of \eqref{eq:h-overlap-short} exact. If row degrees are equal, the cost-weighted
ratio $\mathcal R_{\rm work}$ of \eqref{eq:work-ratio-geometry} equals $1/\sqrt{s_{\rm col}}$ at every gap
level, so a factor coupling $s_{\rm col}$ constraints separates the coordinate and one-block \emph{leading
bounds} by $\sqrt{s_{\rm col}}$; when $s_{\rm col}=O(1)$, sparsity alone delivers no large spectral gain. This compares
upper bounds for one method under two partitions, not a lower bound for any competitor.
\end{proposition}

For a binary column with $s_{\rm col}$ ones, $\mathbf1^\top a_j=s_{\rm col}$ and $\|a_j\|_2^2=s_{\rm col}$, so $q=\mathbf1/s_{\rm col}$
satisfies $A^\top q=\mathbf1$ and $q^\top b=1$; strict feasibility makes every row nonempty. Binary
entries give $H^\star_{ii}=\gamma^{-1}\sum_jA_{ij}x_j^\star=b_i/\gamma=s_{\rm col}/(m\gamma)$, and
\[
H^\star\mathbf1=\gamma^{-1}\sum_jx_j^\star a_j(a_j^\top\mathbf1)=(s_{\rm col}/\gamma)b=(s_{\rm col}^2/(m\gamma))\mathbf1 .
\]
Every row sum of the symmetric nonnegative matrix $H^\star$ equals $s_{\rm col}^2/(m\gamma)$, which is
therefore an eigenvalue and bounds every eigenvalue in absolute value, so
$\|H^\star\|_2=s_{\rm col}^2/(m\gamma)$. The trace is $s_{\rm col}/\gamma$, giving $d_H=m/s_{\rm col}$; since $\rho_i=1$ and
$\rho^2=s_{\rm col}$, also $d_A=m/s_{\rm col}$. The singleton overlap bound equals $b_i/\gamma=H^\star_{ii}$, and for
the full block $\omega_{Bj}=s_{\rm col}$ gives $s_{\rm col}\max_ib_i/\gamma=s_{\rm col}^2/(m\gamma)=\|H^\star\|_2$: both baselines
are available without spectral information about the unknown optimizer. For every gap level
$\ell_{\rm full}(\Delta)=s_{\rm col}\,\ell_i(\Delta)$, and with equal row degrees $d_i=d$, $d_{\rm full}=md$,
substitution gives $\mathcal R_{\rm work}=1/\sqrt{s_{\rm col}}$. The implemented constants are
\[
\kappa_i=2\log2\,s_{\rm col}/(m\gamma),\quad \nu_i=2/\gamma^2,\qquad
\kappa_{\rm f}=2\log2\,s_{\rm col}^2/(m\gamma),\quad \nu_{\rm f}=2s_{\rm col}/\gamma^2,
\]
so $G_0=G_1=1/\sqrt{s_{\rm col}}$ in Proposition~\ref{prop:work-coefficients}. Initialization and finite phase
overhead do not scale by this factor and remain in Theorem~\ref{thm:work-short}.

\subsection{A fully specified construction}
Let $n=m$, index rows and columns modulo $m$, and set
\begin{equation}\label{eq:cyclic-incidence}
A_{ij}=\mathbf1\{(i-j)\bmod m\in\{0,\dots,s_{\rm col}-1\}\} .
\end{equation}
Every row and column has $s_{\rm col}$ ones and $x^\star=\mathbf1/m$ gives $b=(s_{\rm col}/m)\mathbf1$. For any fixed
finite $\lambda^\star$, choosing $c=A^\top\lambda^\star-\gamma\log x^\star$ makes the KKT equations
hold at that optimizer. Repeating each column equally often gives larger $n$ with the same formulas
after splitting the primal mass. Taking $\gcd(m,s)=1$ and $s<m$ gives a full-rank square
construction when that is useful, though full rank is not needed for
Proposition~\ref{prop:balanced-family}.

\subsection{A polynomial gap between the two accountings}\label{app:heterogeneous-cost}
Let $t\ge2$ be an integer, $m=2$, $n=t^2+1$ and $\gamma=1$. The first $t^2$ columns of $A$ equal
$e_1$ and the last equals $e_2$. Set
\[
b=(t^{-4},1-t^{-4})^\top,\qquad x_j^\star=t^{-6}\ (j\le t^2),\qquad x_n^\star=1-t^{-4},\qquad c=-\log x^\star .
\]
Then $q=(1,1)^\top$ certifies total mass, $x^\star>0$ is feasible, and $\lambda^\star=0$ is an SE
minimizer. From the common start $z=(1,1)^\top$ the supplied radius is $r=\sqrt2$ and $F(z)=e-2$,
both independent of $t$. The exact singleton constants and costs are
$\kappa_i=2\log2\,b_i$, $\nu_i=2$, $d_1=t^2+1$, $d_2=2$, so
$C_0=\Theta(1)$, $C_1=\Theta(1)$, $S_0=\Theta(1)$ and $S_1=\Theta(t^2)$, whereas
$(\max_id_i)C_0=\Theta(t^2)$. At target $\eps=t^{-6}$, Theorem~\ref{thm:work-short} including
initialization and output gives $\E W_{\rm total}=\widetilde O(t^3)$, while multiplying the query
bound by the worst block cost gives only $\widetilde O(t^5)$; all transient and preparation terms
are $\widetilde O(t^2)$ and do not conceal the difference. This separates two \emph{accounting
rules for the same method}; it is not a lower bound on another algorithm, and the example has
simple structure that a specialized solver could exploit.

\subsection{Tightness of the subclass}
The range--moment bound need not be tight for signed rows. Append a mass row to $(-1,0,1)$ and set
$x^\star=(t,1-2t,t)$ with $0<t<1/2$. For the signed row $b_i=0$ and
$H^\star_{ii}=2t/\gamma$ while $\bar H^{\rm rng}_i=1/\gamma$, so the ratio $1/(2t)$ is
unbounded. Choosing $c=-\gamma\log x^\star$ makes $\lambda^\star=0$ an attained SE minimizer and all
assumptions hold. This is not a counterexample to the work bound; it shows why effective rank alone
cannot guarantee a comparable curvature sum in the implementation. More generally, if
$H^\star_{ii}>0$ and $\bar H_i\le\varkappa H^\star_{ii}$ for all active coordinates with
$\varkappa\ge1$, then $2\log2\,\bar H_i+\nu_i\Delta\le\varkappa\,\ell_i(\Delta)$, so against the
full envelope the weighted numerator loses at most $\sqrt\varkappa$. The indicator-row
corollary has $\varkappa=1$; no uniform finite $\varkappa$ holds for arbitrary signed data.

\section{Initial gap and distance bounds}\label{app:input-certificates}

\paragraph{Obtaining $r$ and $\bar\Delta$.}
Neither input is free. A strictly feasible $\bar x>0$ gives $\bar\Delta=\phi_\se(z)+P(\bar x)$ by
weak duality, $P$ being the objective of \eqref{eq:primal-short}. The radius is harder:
Proposition~\ref{prop:input-cert} derives one from $\bar x$ and a lower bound on the smallest
nonzero singular value of $A$, but it is conservative (Section~\ref{sec:limits}), and obtaining both inputs
is preprocessing not counted in $W_{\rm init}$.
Let $P(x)=\ip{c}{x}+\gamma\sum_jx_j(\log x_j-1)$ be the primal objective of
\eqref{eq:primal-short}. Producing a feasible point is additional information and the statement
below does not assume it comes for free.

\begin{proposition}[Bounds from strict feasibility]\label{prop:input-cert}
Suppose $\bar x>0$ with $A\bar x=b$ and $\mathbf1^\top\bar x=1$ is supplied. Define
\[
B_z:=\phi_\se(z)+P(\bar x),\qquad x_{\min}:=\min_j\bar x_j>0,\qquad \bar t_j:=1+\frac{B_z}{\gamma\bar x_j} .
\]
Then $B_z\ge F_\se(z)\ge0$, so $\bar\Delta=B_z$ is valid. Let
$0<\widehat\sigma\le\sigma_{\min}^+(A)$ be a verified lower bound on the smallest nonzero singular
value. There exists an SE minimizer $\lambda^\star$ with the same orthogonal projection onto
$\ker A^\top$ as $z$ and
\begin{equation}\label{eq:initial-radius-certificate}
\|z-\lambda^\star\|_2\le r_{\rm cert}:=\frac{2\gamma}{\widehat\sigma}\|\bar t\|_2
\le\frac{2\gamma\sqrt n}{\widehat\sigma}\Bigl(1+\frac{B_z}{\gamma x_{\min}}\Bigr).
\end{equation}
The first bound also controls the diameter of the initial SE sublevel restricted to that gauge
slice. If $B_z=0$, return $z$.
\end{proposition}
\begin{proof}
For $t_j(\lambda)=(a_j^\top\lambda-c_j)/\gamma-\log\bar x_j$, feasibility of $\bar x$ gives the
exact identity
\begin{equation}\label{eq:feasible-reference-identity}
\phi_\se(\lambda)+P(\bar x)=\gamma\sum_j\bar x_j\,h(t_j(\lambda)),
\end{equation}
and weak duality gives $B_z\ge F_\se(z)$. On $\{\phi_\se(\lambda)\le\phi_\se(z)\}$ each nonnegative
summand is at most $B_z$. We have $|t|\le1+h(t)$: for $t<0$ from $h(t)\ge-t-1$,
and for $t\ge0$ from $h(t)\ge t^2/2$ and $\sqrt{2h}\le1+h$. Hence $|t_j(\lambda)|\le\bar t_j$ on the
sublevel.
The dual is invariant along $\ker A^\top$ because $b=A\bar x$ is orthogonal to that kernel, so shift
an attained minimizer inside the kernel to match the kernel component of $z$. Then
$z-\lambda^\star\in\operatorname{range}A$, where the singular-value bound applies, and both
endpoints lie in the initial sublevel:
\[
\widehat\sigma\|z-\lambda^\star\|_2\le\|A^\top(z-\lambda^\star)\|_2
=\gamma\|t(z)-t(\lambda^\star)\|_2\le2\gamma\|\bar t\|_2 .
\]
The same argument applies to any two sublevel points in the chosen gauge, and the second inequality
follows from $\bar x_j\ge x_{\min}$.
\end{proof}
This SE gauge removes only $\ker A^\top$; it is \emph{not} the larger LSE invariance space
$\mathcal G$, since $\phi_\se$ is not invariant in the mass direction $q$. No gauge projection is
needed in the sparse iterations; the theorem only needs a valid initial distance to some
minimizer.

\paragraph{A spectral bound for the cyclic family.}
For the square cyclic matrix \eqref{eq:cyclic-incidence}, the Fourier singular values are $s_{\rm col}$ and
$|\sin(\pi s_{\rm col}k/m)/\sin(\pi k/m)|$ for $k=1,\dots,m-1$. If $1\le s_{\rm col}<m$ and
$\gcd(m,s_{\rm col})=1$, then $s_{\rm col}k\not\equiv0$, every numerator is at least $\sin(\pi/m)$ in magnitude, and every denominator is at
most one, so $\widehat\sigma=\sin(\pi/m)$ is valid. The uniform feasible point has $x_{\min}=1/m$, giving
\[
r_{\rm cert}\le\frac{2\gamma\sqrt m}{\sin(\pi/m)}\Bigl(1+\frac{mB_z}{\gamma}\Bigr).
\]
This shows computability without a dense SVD; it is not a sharp practical radius, and on this
family it exceeds the true distance by a factor that grows with $m$; see
Section~\ref{sec:limits}. The known-solution protocol of Appendix~\ref{app:exp-full} instead
supplies the exact distance to the constructed minimizer.

\paragraph{Preprocessing.}
Given $\bar x$ and $\widehat\sigma$, forming these bounds costs one data pass and vector
operations. Finding $\bar x$ and verifying $\widehat\sigma$ can be considerably more expensive;
denote that work by $W_{\rm geometry}$ and add it to
$W_{\rm init}+W_{\rm opt}+W_{\rm output}$ when the inputs are not provided. A small feasibility
margin $x_{\min}$ or a small spectral gap makes \eqref{eq:initial-radius-certificate} very conservative:
a finite bound is not a guarantee of a well-conditioned instance.

\section{Fixed-support validation}\label{app:fixed-support-protocol}
This specifies \emph{additional validation}, not completed optimization benchmarks. The generator
below is deterministic and what it feeds are numerical checks of the algebraic identities; no new
convergence curves are produced.

Fix one binary matrix \eqref{eq:cyclic-incidence}, a $\gamma>0$ and an index $j_0$. For
$0\le\vartheta<1$ set
\begin{equation}\label{eq:fixed-support-generator}
x^\star(\vartheta)=(1-\vartheta)\mathbf1/m+\vartheta e_{j_0},\qquad b(\vartheta)=Ax^\star(\vartheta),\qquad
c(\vartheta)=-\gamma\log x^\star(\vartheta).
\end{equation}
The matrix and its support never change, strict positivity holds, and $\lambda^\star=0$ is an exact
SE minimizer throughout. Use a common nonoptimal start, for example $z=\gamma e_{i_0}$ on a nonempty
row, for both the singleton and one-block methods; the supplied radius $r=\|z\|_2=\gamma$ is then
valid for every $\vartheta$. A zero start would already solve the generated problem and must not be used.

For every $\vartheta$ the coordinate constant is exact by Corollary~\ref{cor:indicator-exact}, and at
$\vartheta=0$ the balanced formulas apply. As $\vartheta\uparrow1$,
\[
H^\star(\vartheta)\longrightarrow a_{j_0}a_{j_0}^\top/\gamma,\qquad d_H(\vartheta)\longrightarrow1,
\qquad d_A=m/s\ \text{ for all }\vartheta,
\]
so the protocol sweeps $d_H$ across its whole range while holding sparsity, access costs and the
radius fixed. Do not use $\vartheta=1$, which violates strict positivity. A computable full-block
constant is $\bar H_{\rm f}(\vartheta)=s\max_ib_i(\vartheta)/\gamma$; away from balance it need not equal
$\|H^\star(\vartheta)\|_2$.

Record $d_H$, $d_A$, $\Delta_{\rm cross}$, the exact and computed curvature sums,
$\bar H_{\rm f}/\|H^\star\|_2$, and the weighted work factors. For actual optimization runs also
record the starting gap, radius input, common absolute target and residual, seeds, phase
comparisons and all data-access work. Spectral diagnostics are computed offline, their cost
reported separately, and must not be silently handed to an otherwise bound-only algorithm. Note
that the protocol changes $b$, $c$ and in general the initial gap, so those changes should be
reported rather than claiming the spectrum is the only thing that moved. A separate ablation should
vary the association between row cost $d_i$ and baseline curvature at fixed aggregate sparsity,
reporting $S_0=\sum_id_i\sqrt{\kappa_i}$ alongside $(\max_id_i)\sum_i\sqrt{\kappa_i}$.

\section{Resource comparison and a sharper \texorpdfstring{$L_1$}{L1} target}\label{app:resource-comparison}
\subsection{Comparing native guarantees}
In Table~\ref{tab:oracle-resources}, $R$ is a distance to a solution, $\widetilde R$ a sublevel
radius in the cited coordinate analysis, and $r$ a supplied upper bound used in our schedule. Our row describes expectation after fixed-anchor repetition,
not a deterministic guarantee.

\paragraph{The cost of a frozen constant.}
The comparison the introduction states in one line is this. The method already takes an input
$\bar\Delta\ge F(z)$, and the interpolation argument of Appendix~\ref{app:acc-proof} confines the
accelerated iterate to $\{F\le10\bar\Delta/\delta\}$, so
$\kappa_B+10\nu_B\bar\Delta/\delta$ is a legitimate coordinate Lipschitz constant, fixed before the
first iteration and requiring no schedule of ours. Feeding it to \citet{allenzhu2016} gives
$N=\widetilde O\bigl(r(C_0+\sqrt{\bar\Delta}C_1)/\sqrt\eps\bigr)$, against
$N=\widetilde O\bigl(1+rC_0/\sqrt{\delta\eps}+rC_1/\sqrt\delta\bigr)$ in
Theorem~\ref{thm:phase-short}. The accuracy channel $rC_0/\sqrt\eps$ is identical; the growth
channel is charged $\sqrt{\bar\Delta/\eps}$ times more by the frozen constant. Since $\bar\Delta$ is
an input that a user can only over-estimate, and the runs of Section~\ref{sec:exp-short} show the
envelope exceeding the curvature actually met by up to $2.9\cdot10^4$, that factor is paid on the
instances we run and not only in the bound.

\paragraph{Restarts, and what the continuous schedule does not buy.}
A single frozen constant is not the sharpest alternative; restarts are. Run the method in phases $p=0,1,\dots$, each targeting half the gap of the last, and
inside phase $p$ freeze the constants at that phase's certified level
$\kappa_B+10\nu_B\Delta_p/\delta$ with $\Delta_p=2^{-p}\bar\Delta$, which is the level the same
interpolation argument supplies at the phase's own gap bound. Lemma~\ref{lem:presampling} applies
verbatim with $M$ frozen: it only needs $M^2\ge\mathcal T(10\Delta_p/\delta)$ on the phase's sublevel
set, which the phase target guarantees. Summing the per-phase counts of
Theorem~\ref{thm:phase-short} over the geometric schedule gives

\begin{proposition}[Restarted ACD]\label{prop:restart}
Let $D$ bound the diameter of $\{F\le\bar\Delta\}$ in the gauge slice. Restarted ACD with per-phase
frozen constants reaches $\E F\le\eps$ in
$\widetilde O\bigl(DC_0/\sqrt{\eps}+DC_1\bigr)$ block queries.
\end{proposition}

\noindent This is \eqref{eq:query-short} with $r$ replaced by $D$, and both channels match. The
comparison therefore turns entirely on whether a radius smaller than the sublevel diameter is
available, and here we must be plain: the only radius we can compute,
Proposition~\ref{prop:input-cert}, bounds exactly that diameter. On the inputs a user can actually
supply, the continuous schedule does not improve on restarts; what it removes is the phase
bookkeeping -- re-validating the anchor, recomputing the sampling tables and discarding the
momentum at every boundary -- not a factor in the rate. We therefore present it as restart-free
rather than as faster, and leave open whether an instance family with a computable $r\ll D$ exists.
The geometric schedule is out of reach at the budgets of Section~\ref{sec:exp-short}: with the proved
constants, certifying the first phase target needs $T_k=4\Pi/\bar\Delta$, and $3000$ epochs cover a
quarter of that distance at $m=300$ and a sixth at $m=1000$, so the runs tagged \texttt{ACD:restart}
complete a single phase and carry no comparison between the two schedules.

\begin{table}[htbp]
\centering\small
\caption{Native gap-accuracy bounds and the extra resources each method consumes. Logarithms and
absolute constants are suppressed. A segment query is not a unit-cost partial derivative.}
\label{tab:oracle-resources}
\begin{tabular}{@{}>{\raggedright\arraybackslash}p{0.18\linewidth}p{0.38\linewidth}>{\raggedright\arraybackslash}p{0.37\linewidth}@{}}
\toprule
Method & Native update count & Information and auxiliary work \\
\midrule
Ordinary ACD \citep{allenzhu2016} &
$\widetilde O(R\sum_i\sqrt{L_i}/\sqrt\eps)$ &
Fixed coordinate smoothness; sparse accelerated state when the problem supports it. \\[1ex]
\citet[Theorem~8]{lobanov2026} &
$\widetilde O\bigl(\widetilde R(\sum_i\sqrt{H_{0,i}}/\sqrt\eps+\sum_i\sqrt{H_{1,i}})\bigr)$ &
Coordinate gap model; known $f^\star$; segment relaxation. Phase-wise sampling and verified
doubling available. \\[1ex]
AGMsDR \citep{nesterov2021relaxation,vankov2025} &
$\widetilde O(R\sqrt{L_0/\eps}+(L_1R)^{2/3})$ &
Global gradient-dependent model; full gradients and function differences; segment minimization
costs extra oracle calls. \\[1ex]
This paper &
$\widetilde O(1+rC_0/\sqrt\eps+rC_1)$ &
BGG constants, supplied $r,\bar\Delta$; one sampled block gradient; no segment solve and no
current-gap evaluation inside a phase. \\
\bottomrule
\end{tabular}
\end{table}

Lobanov's extensions retain a known $f^\star$ while removing the sublevel radius from the inputs by
verified doubling, and inexact relaxation and phase-wise sampling belong in any fair comparison
\citep{lobanov2026}. Neither fixing the parameters beforehand nor non-uniform square-root weights
distinguishes our method; the claim concerns the missing auxiliary solve and the sparse
implementation of the fixed-constant schedule, not a better native exponent. The bridge between the
Hessian condition and a local model is Lemma~\ref{lem:hessian-local-radius}.

The full-gradient result of \citet{vankov2025} motivates a sharper coordinate $L_1$ dependence, but
its proof cannot be imported by replacing a full gradient with an unbiased coordinate estimator:
its monotonicity, segment optimality and accumulated-descent arguments would all have to survive
that replacement. Function-value and directional-derivative costs on generic affine sums should be
counted by touched factors rather than hidden behind an update count.

\paragraph{Relation to accelerated alternating minimization.}
\citet{guminov2021combination} give an accelerated alternating-minimization method with a $k^{-2}$
rate, which on the entropic transport dual is accelerated block coordinate ascent over the two
marginal families. It is the natural point of comparison for this paper on the LSE side, and the
comparison is instructive precisely because the two methods buy different things. Their per-block
step is an exact minimization, which is affordable only when the block subproblem has closed form;
ours is a fixed-curvature step, which needs no subproblem solver at all but pays the envelope. Their
iteration touches an entire marginal, so the per-iteration cost is a full pass over the data and no
sparsity claim is made; ours touches $d_B$ nonzeros. And their rate is stated for the normalized
dual, where the Hessian is bounded, so the difficulty our Theorem~\ref{thm:hg-short} addresses does
not arise. A fair head-to-head would need a family carrying both a block structure with solvable
subproblems \emph{and} enough sparsity for the access cost to matter; transport has the first and
not usefully the second, and our synthetic family the reverse. We therefore report the structural
comparison that transport does allow --- Sinkhorn against a generic sparse method,
Section~\ref{sec:exp-short} --- and do not claim a comparison we have not run.

\paragraph{Relation to relative smoothness.}
The SE objective is relatively smooth with respect to a suitable Bregman reference, and the
Bregman literature \citep{bauschke2017descent,lu2018relatively,hanzely2021bregman} supplies methods
that avoid a global Lipschitz constant by a different route. Two things separate the present
analysis. First, accelerated Bregman methods require a triangle-scaling or symmetry condition on
the reference function, and their optimal rates are constrained by the lower bounds of
\citet{dragomir2021optimal}; we instead keep the Euclidean geometry and move the difficulty into a
gap-dependent constant. Second, and this is the binding objection, a Bregman proximal step on the SE dual
does not decompose over sparse blocks in a way that preserves the exact lazy representation of
Section~\ref{sec:lazy}, which is the property the arithmetic bound rests on. A proper comparison of
the two routes on entropy duals is worth doing and is not attempted here.

\subsection{A proved aggregation inequality and an unproved extension}\label{app:l1-target}
\begin{lemma}[Gradient-sensitive square-root aggregation]\label{lem:l1-holder}
For $L_{0,i},L_{1,i}\ge0$, set $\mathcal A_0=\sum_i\sqrt{L_{0,i}}$ and
$\mathcal B_1=\sum_iL_{1,i}^{2/3}$. Every $g\in\R^m$ satisfies
\begin{equation}\label{eq:gradient-holder}
\Bigl(\sum_i\sqrt{L_{0,i}+L_{1,i}|g_i|}\Bigr)^2\le2\mathcal A_0^2+2\mathcal B_1^{3/2}\|g\|_2 .
\end{equation}
\end{lemma}
\begin{proof}
By H\"older with exponents $4/3$ and $4$,
$\sum_i\sqrt{L_{1,i}|g_i|}\le(\sum_iL_{1,i}^{2/3})^{3/4}(\sum_i|g_i|^2)^{1/4}$. Apply
$\sqrt{x+y}\le\sqrt x+\sqrt y$ termwise, then $(a+b)^2\le2a^2+2b^2$.
\end{proof}

\paragraph{An unproved target.}
Under \eqref{eq:lc-model}, \eqref{eq:gradient-holder} suggests seeking a method with query
complexity $\widetilde O(1+r\mathcal A_0/\sqrt\eps+r^{2/3}\mathcal B_1)$. For equal $L_{1,i}=L_1$
its growth contribution would be $m(L_1r)^{2/3}$ instead of the $mL_1r$ that the established
gap-based reduction gives, an improvement in the large-$L_1r$ regime only. A plausible sufficient
ingredient would be a one-phase bound
\begin{equation}\label{eq:l1-phase-target}
\E F(u_N)\le K\frac{r^2\mathcal A_0^2}{N^2}+K\frac{r^2\mathcal B_1^3}{N^3}\Delta,
\qquad F(z)\le\Delta,
\end{equation}
with universal $K$, an admissible sampling and acceptance scheme, and an appropriate radius
invariant. Equation~\eqref{eq:l1-phase-target} is a target, not a lemma of this paper, and is not
established for Algorithm~\ref{alg:main-short}; even if proved for one phase it would still need a
separate radius argument before moving-anchor restarts. Computing every $|g_i|$ to set
probabilities would cost a full gradient, so the aggregation lemma by itself provides neither a
sparse sampling implementation nor the required stochastic monotonicity. The proved query bound
remains Theorem~\ref{thm:phase-short} and its specialization through Lemma~\ref{lem:lc-to-hc}.

\subsection{Limitations in detail}\label{app:limits-full}
\paragraph{Looseness of the supplied radius.} Proposition~\ref{prop:input-cert} bounds the distance
to the optimum by $\gamma\sqrt n\,\widehat\sigma^{-1}(1+B_z/(\gamma x_{\min}))$ with
$x_{\min}=\min_j\bar x_j$. On the balanced family of Proposition~\ref{prop:balanced-family} this
grows polynomially in $m$ while the true distance stays $O(\gamma)$, so feeding the bound to
Algorithm~\ref{alg:main-short} would inflate every phase length; the runs therefore supply an oracle
radius, and Table~\ref{tab:cert-sweep} shows the slack growing with $r$.

\paragraph{The dynamic range of the SE iteration.} With
$e^{O(\|A^\top(\lambda-\lambda^\star)\|_\infty/\gamma)}$ as the range of the unnormalized factors,
overflow sets in once $\|A^\top(\lambda-\lambda^\star)\|_\infty/\gamma\gtrsim710$, a displacement of
about $7$ at $\gamma=10^{-2}$ for unit entries and much less for the entries of size $\approx m/s$
in our family -- long before the iterate is inaccurate in any meaningful sense. The LSE dual is immune
because $p(\lambda)$ is shift invariant and can be formed in the log domain; that asymmetry, not
the rate, is what keeps small-$\gamma$ transport off the SE route.

\paragraph{Scope of the separations.} The $\sqrt{s_{\rm col}}$ gain of
Proposition~\ref{prop:balanced-family} and the polynomial gain of
Appendix~\ref{app:heterogeneous-cost} both hold between \emph{our own} upper bounds under two
accountings of the same method. They are not runtime claims, and they are not lower bounds against
any competitor; Appendix~\ref{app:resource-comparison} states what a fair comparison of native
guarantees would need instead.

\end{document}